\documentclass{amsart}

\usepackage{amssymb,amsmath,amsthm,latexsym,booktabs,youngtab,tikz,mdframed}

\usetikzlibrary{positioning,arrows,calc}

\mdfdefinestyle{mystyle}{ innerleftmargin = -7 mm, leftmargin = 5 mm, rightmargin = 5 mm }

\theoremstyle{definition}
\newtheorem{definition}{Definition}[section]
\newtheorem{example}[definition]{Example}
\newtheorem{remark}[definition]{Remark}
\newtheorem{problem}[definition]{Problem}

\theoremstyle{plain}
\newtheorem{lemma}[definition]{Lemma}
\newtheorem{proposition}[definition]{Proposition}
\newtheorem{theorem}[definition]{Theorem}
\newtheorem{corollary}[definition]{Corollary}
\newtheorem{conjecture}[definition]{Conjecture}

\newcommand{\sgn}{\epsilon}
\newcommand{\esgn}{\xi}
\newcommand{\Epsilon}{\mathcal{E}}
\newcommand{\nn}{\!\!}
\newcommand{\nnn}{\!\!\!}

\allowdisplaybreaks

\begin{document}

\title[The symmetric group and polynomial identities]
{Structure theory for the group algebra of\\the symmetric group,
with applications to polynomial identities for the octonions}

\author[Bremner]{Murray R. Bremner}

\address{Department of Mathematics and Statistics, University of Saskatchewan, Canada}

\email{bremner@math.usask.ca}

\author[Madariaga]{Sara Madariaga}

\address{Department of Mathematics and Statistics, University of Saskatchewan, Canada}

\email{madariaga@math.usask.ca}

\author[Peresi]{Luiz A. Peresi}

\address{Instituto de Matem\'atica e Estat\'istica, Universidade de S\~ao Paulo, Brazil}

\email{peresi@ime.usp.br}



\begin{abstract}
In part 1, we review the structure theory of $\mathbb{F} S_n$, the group algebra of
the symmetric group $S_n$ over a field of characteristic 0.
We define the images $\psi(E^\lambda_{ij})$ of the matrix units $E^\lambda_{ij}$ ($1 \le i, j \le d_\lambda$),
where $d_\lambda$ is the number of standard tableaux of shape $\lambda$,
and obtain an explicit construction of Young's isomorphism
$\psi\colon \bigoplus_\lambda M_{d_\lambda}(\mathbb{F}) \to \mathbb{F} S_n$.
We then present Clifton's algorithm for the construction of the representation matrices
$R^\lambda(p) \in M_{d_\lambda}(\mathbb{F})$ for all $p \in S_n$, and obtain the reverse isomorphism
$\phi\colon \mathbb{F} S_n \to \bigoplus_\lambda M_{d_\lambda}(\mathbb{F})$.

In part 2, we apply the structure theory of $\mathbb{F} S_n$ to the study of multilinear polynomial identities 
of degree $n \le 7$ for the algebra $\mathbb{O}$ of octonions over a field of characteristic 0.
We compare our results with earlier work of Racine, Hentzel \& Peresi, and Shestakov \& Zhukavets on
the identities of degree $n \le 6$.
We use computational linear algebra to verify that every identity in degree 7 is a consequence of the known identities
of lower degrees: there are no new identities in degree 7.
We conjecture that the known identities of degree $\le 6$ generate all octonion identities in characteristic 0.
\end{abstract}

\dedicatory{To our colleague Irvin Roy Hentzel on his 71st birthday}

\maketitle

\tableofcontents

%%%%%%%%%%%%%%%%%%%%%%%%%%%%%%%%%%%%%%%%%%%%%%%%%%%%%%%%%%%%%%%%%%%%%%%%%%%%%%%%%%%%%%%%%%%%%%%%%%%%%%%%%%%%

\section{Structure theory for the group algebra of the symmetric group}

In this first part, we study the structure of the group algebra $\mathbb{F} S_n$ of the symmetric group $S_n$ 
on $n$ letters.
As a vector space over $\mathbb{F}$, $\mathbb{F} S_n$ has basis 
$\{ \sigma \mid \sigma \in S_n \}$, and the associative multiplication is defined on basis elements by 
the product in $S_n$ and extended bilinearly.
We assume throughout that $\mathbb{F}$ is a field of characteristic 0.

By the classical structure theory of associative algebras, we know that $\mathbb{F} S_n$ is semisimple, 
and hence isomorphic to the direct sum of full matrix algebras with entries in division algebras over $\mathbb{F}$.
In fact, each of these division algebras is isomorphic to $\mathbb{F}$, and the Wedderburn decomposition 
is given by two isomorphisms,
  \begin{equation}\label{iso}
  \tag{W}
  \phi \colon \mathbb{F} S_n \, \longrightarrow \, \bigoplus_\lambda M_{d_\lambda}(\mathbb{F}),
  \qquad
  \psi \colon \bigoplus_\lambda M_{d_\lambda}(\mathbb{F}) \, \longrightarrow \, \mathbb{F} S_n,  
  \end{equation}
where the sum is over all partitions $\lambda$ of $n$, and $d_\lambda$ is the dimension of 
the irreducible representation of $S_n$ corresponding to $\lambda$.

The matrices obtained by restricting $\phi$ to $S_n$, and taking the component of $\phi$ for partition $\lambda$,
have entries in $\{0,\pm1\}$ and form the natural representation of $S_n$.
We will show how to efficiently compute these matrices for all $\lambda$ and all
$p \in S_n$.

Each matrix algebra $M_{d_\lambda}(\mathbb{F})$ has a basis of matrix units $E^\lambda_{ij}$ for
$i, j = 1, \dots, d_\lambda$ which multiply according to the standard relations,
\[
E^\lambda_{ij} E^\mu_{k\ell} = \delta_{\lambda\mu} \delta_{jk} E^\lambda_{i\ell}.
\]
The isomorphism $\psi$ produces elements $\psi(E^\lambda_{ij})$ in $\mathbb{F} S_n$
which obey the same equations.
We show how to calculate these elements of $\mathbb{F} S_n$.

None of the material in this first part is original.
We compiled the results from many sources, and attempted to make the terminology more contemporary and the notation
simpler and more consistent.
The structure theory of $\mathbb{F} S_n$ was original worked out by Young \cite{Young1977}.
The proofs in Young's papers were simplified by Rutherford \cite{Rutherford1948}, and
the theory was reformulated in more modern terminology and notation by Boerner \cite{Boerner1963},
following suggestions by von Neumann and van der Waerden \cite{vanderWaerden1970}.
A substantial simplification of the algorithms for computing the matrices in the natural representation 
(the isomorphism $\phi$) was introduced by Clifton \cite{Clifton1980,Clifton1982}.
Our exposition is based on the Ph.D.~thesis of Bondari \cite{Bondari1993,Bondari1997}.

%%%%%%%%%%%%%%%%%%%%%%%%%%%%%%%%%%%%%%%%%%%%%%%%%%%%%%%%%%%%%%%%%%%%%%%%%%%%%%%%%%%%%%%%%%%%%%%%%%%%%%%%%%%%

\subsection{Young diagrams and tableaux}

We start by giving the definitions and elementary properties of the basic objects in the theory.
The symmetric group $S_n$ is the group of all permutations of the set $\{1, \dots, n \}$.
We write $\lambda \vdash n$ to indicate that $\lambda$ is a partition of $n$; that is,
$\lambda = ( n_1, \dots, n_k )$ where $n = n_1 + \cdots + n_k$ and $n_1 \ge \cdots \ge n_k \ge 1$.
If $n \le 9$ then we write unambiguously $\lambda = n_1 \cdots n_k$.

\begin{definition}
The \textbf{Young diagram} $Y^\lambda$ of the partition $\lambda = ( n_1, \dots, n_k )$ consists of
$k$ left-justified rows of empty square boxes where row $i$ contains $n_i$ boxes.
\end{definition}

\begin{example}
Young diagrams for some partitions of $n = 9$:
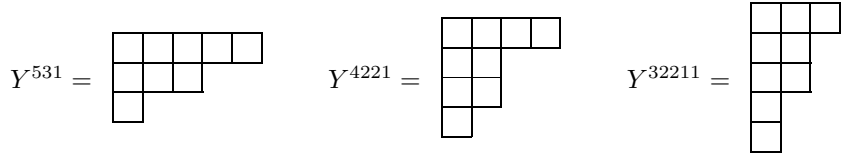
\[
Y^{531}
=
\begin{array}{c}
\yng(5,3,1)
\end{array}
\qquad
Y^{4221}
=
\begin{array}{c}
\yng(4,2,2,1)
\end{array}
\qquad
Y^{32211} =
\begin{array}{c}
\yng(3,2,2,1,1)
\end{array}
\]
\end{example}

\begin{definition}
Suppose that $\lambda = ( n_1, \dots, n_k )$ and $\lambda' = ( n'_1, \dots, n'_\ell )$ are partitions of $n$.
We say that $\lambda \prec \lambda'$ (equivalently, $Y^\lambda \prec Y^{\lambda'}$) if and only if
either $n_1 < n'_1$ or there exists $i \ge 1$ such that $n_1 = n'_1$, \dots, $n_i = n'_i$ but $n_{i+1} < n'_{i+1}$.
\end{definition}

\begin{example}
The seven Young diagrams for $n = 5$ in decreasing order:
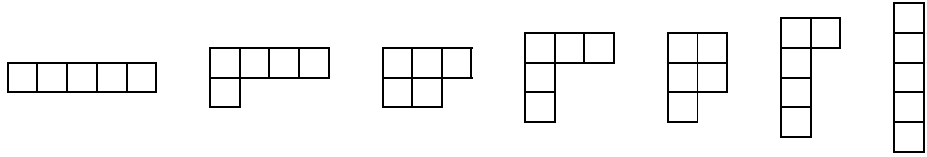
\[
\begin{array}{c} \yng(5) \end{array}\quad
\begin{array}{c} \yng(4,1) \end{array}\quad
\begin{array}{c} \yng(3,2) \end{array}\quad
\begin{array}{c} \yng(3,1,1) \end{array}\quad
\begin{array}{c} \yng(2,2,1) \end{array}\quad
\begin{array}{c} \yng(2,1,1,1) \end{array}\quad
\begin{array}{c} \yng(1,1,1,1,1) \end{array}\quad
\]
\end{example}

\begin{definition}
A \textbf{Young tableau} $T^\lambda$ of shape $\lambda$ where $\lambda \vdash n$
consists of a bijective assignment of the numbers $1, \dots, n$ to the boxes in the Young diagram $Y^\lambda$.
The number in row $i$ and column $j$ will be denoted $T(i,j)$.
The sequence of numbers from left to right in row $i$ will be denoted $T(i,-)$;
the sequence of numbers from top to bottom in column $j$ will be denoted $T(-,j)$.
A tableau is \textbf{standard} if all the sequences $T(i,-)$ and $T(-,j)$ are increasing.
\end{definition}

\begin{remark}
The number $d_\lambda$ of standard tableaux for the Young diagram $Y^\lambda$ is given by the hook formula,
  \[
  d_\lambda = \frac{n!}{\prod_{i,j} |h_{ij}|},
  \]
where $|h_{ij}|$ is the number of boxes in the hook with corner at position $(i,j)$:
  \[
  h_{ij} = \{ \, (i,j') \mid j \le j' \, \} \cup \{ \, (i',j) \mid i \le i' \, \}.
  \]
Another way to write this formula which is easier to implement on a computer is
  \[
  d_\lambda = n! \, \frac{\prod_{i<j}(m_i-m_j)}{\prod_i m_i!},
  \]
where $m_i = n_i + k - i$ ($i=1,\dots,k$) for $\lambda = ( n_1, \dots, n_k )$; see \cite[Theorem 4.2]{Boerner1963}.
For a detailed discussion of the hook formula, see \cite[\S 5.1.4, Theorem H]{Knuth1998}.
For a modern bijective proof of the hook formula, see \cite{NPS1997}.
\end{remark}

\begin{definition}
Given two tableaux $T$ and $T'$ of shape $\lambda \vdash n$, let $i$ be the least row index for which
$T(i,-) \ne T'(i,-)$, and let $j$ be the least column index for which $T(i,j) \ne T'(i,j)$.
The \textbf{lexicographical order} (lex order) on tableaux is defined by $T \prec T'$ if and only if $T(i,j) < T'(i,j)$.
\end{definition}

\begin{example}
The standard tableaux for $n = 5$, $\lambda = 32$ in lex order:
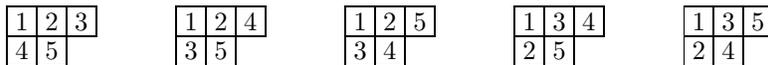
\[
\young(123,45) \qquad\quad
\young(124,35) \qquad\quad
\young(125,34) \qquad\quad
\young(134,25) \qquad\quad
\young(135,24)
\]
\end{example}

\begin{definition}
For each partition $\lambda \vdash n$, the group $S_n$ acts on the tableaux of shape $\lambda$ 
by permuting the numbers in the boxes.
For $p \in S_n$ and tableau $T$, the result will be denoted $pT$:
that is, if $T(i,j) = x$ then $(pT)(i,j) = p  x$.
\end{definition}

%%%%%%%%%%%%%%%%%%%%%%%%%%%%%%%%%%%%%%%%%%%%%%%%%%%%%%%%%%%%%%%%%%%%%%%%%%%%%%%%%%%%%%%%%%%%%%%%%%%%%%%%%%%%

\subsection{Horizontal and vertical permutations}

Each tableau of shape $\lambda \vdash n$ determines certain subgroups of $S_n$ which play an essential role 
in the theory.

\begin{definition}
Given a tableau $T$ of shape $\lambda = (n_1,\dots,n_k) \vdash n$, we write $G_H(T)$ for the subgroup of $S_n$
consisting of all \textbf{horizontal permutations} for $T$.
These are the permutations $h \in S_n$ which leave the rows fixed as sets:
for all $i = 1, \dots, k$, if $x \in T(i,-)$ then $h  x \in T(i,-)$.
Similarly, the subgroup $G_V(T)$ of \textbf{vertical permutations} of $T$ consists of all permutations $v \in S_n$ 
which leave the columns fixed as sets: for all $j = 1, \dots, n_1$, if $x \in T(-,j)$ then $v  x \in T(-,j)$.
\end{definition}

\begin{remark}
If we regard the rows $T(i,-)$ and columns $T(-,j)$ as sets, then $G_H(T)$ and $G_V(T)$ can be defined as direct products:
  \[
  G_H(T) = \prod_{i=1}^k S_{T(i,-)},
  \qquad
  G_V(T) = \prod_{j=1}^{n_1} S_{T(-,j)},
  \]
where $S_X$ denotes the group of all permutations of the set $X$.
\end{remark}

\begin{lemma}
If $T$ is a tableau of shape $\lambda \vdash n$ then $G_H(T) \cap G_V(T) = \{ \iota \}$ where 
$\iota \in S_n$ is the identity permutation.
It follows that if $h, h' \in G_H(T)$ and $v, v' \in G_V(T)$ with $hv = h'v'$ then $h = h'$ and $v = v'$.
\end{lemma}

\begin{proof}
If $hv = h'v'$ then $(h')^{-1} h = v' v^{-1}$ and so both equal $\iota$.
\end{proof}

\begin{figure}[h]
\begin{tikzpicture}
  % first tableau
  \draw
  (0,0) -- (3,0) -- (3,-0.5) -- (2.5,-0.5) -- (2.5,-1) --
  (2,-1) -- (2,-1.5) -- (1.5,-1.5) -- (1.5,-2) -- (1,-2) --
  (1,-2.5) -- (0.5,-2.5) -- (0.5,-3) -- (0,-3) -- (0,0) ;
  \node[draw=none,fill=none] (x2) at ( 1.75,-0.75) {$x$};
  \node[draw=none,fill=none] (x1) at ( 0.75,-1.75) {$x$};
  \node[draw=none,fill=none] at (-0.25,-0.75) {$k$};
  \node[draw=none,fill=none] at (-0.25,-1.75) {$i$};
  \node[draw=none,fill=none] at ( 0.75, 0.25) {$j$};
  \node[draw=none,fill=none] at ( 1.75, 0.25) {$\ell$};
  \node[draw=none,fill=none] at ( 2.75, -2.75) {\fbox{$T$}};
  %second tableau
  \draw (8,0) -- (11,0) -- (11,-0.5) -- (10.5,-0.5) -- (10.5,-1) --
        (10,-1) -- (10,-1.5) -- (9.5,-1.5) -- (9.5,-2) -- (9,-2) --
        (9,-2.5) -- (8.5,-2.5) -- (8.5,-3) -- (8,-3) -- (8,0) ;
  \node[draw=none,fill=none] (px2) at ( 9.75,-0.75) {$px$};
  \node[draw=none,fill=none] (px1) at ( 8.75,-1.75) {$px$};
  \node[draw=none,fill=none] at ( 7.75,-0.75) {$k$};
  \node[draw=none,fill=none] at ( 7.75,-1.75) {$i$};
  \node[draw=none,fill=none] at ( 8.75, 0.25) {$j$};
  \node[draw=none,fill=none] at ( 9.75, 0.25) {$\ell$};
  \node[draw=none,fill=none] at ( 10.75, -2.75) {\fbox{$pT$}};
  % arrows
  \node[draw=none,fill=none] (1) at ( 3.5,-0.75) {\,};
  \node[draw=none,fill=none] (2) at ( 6.75,-0.75) {\,};
  \node[draw=none,fill=none] (3) at ( 3.5,-1.75) {\,};
  \node[draw=none,fill=none] (4) at ( 6.75,-1.75) {\,};
  \path[every node/.style={font=\footnotesize},->]
    (1) edge [bend left] node [above] {$p$} (2);
  \path[every node/.style={font=\footnotesize},->]
    (4) edge [bend left] node [below] {$p^{-1}$} (3);
  \path[every node/.style={font=\footnotesize},->]
    (x1) edge [bend left] node [above] {$q$\;\;\;} (x2);
  \path[every node/.style={font=\footnotesize},->]
    (px1) edge [bend left] node [above] {$pqp^{-1}\;\;\;$} (px2);
\end{tikzpicture}
\caption{Tableaux for the proof of Lemma \ref{lemmaconjugate}}
\label{conjugation}
\end{figure}

\begin{lemma} \label{lemmaconjugate}
Assume that $T$ is a tableau of shape $\lambda \vdash n$ and $p \in S_n$.

(a)
If $h \in G_H(T)$ then $php^{-1} \in G_H(pT)$.
Since conjugation by $p$ is invertible, it is a bijection from $G_H(T)$ to $G_H(pT)$.

(b)
If $v \in G_V(T)$ then $pvp^{-1} \in G_V(pT)$.
Since conjugation by $p$ is invertible, it is a bijection from $G_V(T)$ to $G_V(pT)$.
\end{lemma}

\begin{proof}
We refer to Figure \ref{conjugation}.
Suppose the permutation $q \in S_n$ moves the number $x$ from position $(i,j)$ of the tableau $T$ to position $(k,\ell)$;
this is represented by the arrow labelled $q$ in the left tableau.
Following the lower curved arrow labelled $p^{-1}$, then the arrow in the left tableau labelled $q$,
and finally the upper curved arrow labelled $p$, we see that the permutation $pqp^{-1}$ moves $x' = px$ from 
position $(i,j)$ of the tableau $pT$ to position $(k,\ell)$.
This is represented by the arrow labelled $pqp^{-1}$ in the right tableau.
In particular, if $q = h \in G_H(T)$ then $i = k$, and so $php^{-1}$ is a horizontal permutation for $pT$.
Similarly, if $q = v \in G_V(T)$ then $j = \ell$, and so $pvp^{-1}$ is a vertical permutation for $pT$.
\end{proof}

\begin{remark} \label{remarkconjugate}
The notation $hv T$ indicates that we apply the vertical permutation $v \in G_V(T)$ to $T$ and then apply the horizontal
permutation $h \in G_H(T)$ to $vT$.
However, $h$ may not be a horizontal permutation for $vT$.
We can rewrite this using permutations which are horizontal or vertical for the tableaux on which they act:
we have $hv T = (hvh^{-1}) hT$ and $hvh^{-1}$ is a vertical permutation for $hT$.
\end{remark}

\subsection{Row and column intersections}

The next few results investigate the intersection $T(i,-) \cap T'(-,j)$ for tableaux $T$ and $T'$ of shapes 
$\lambda$ and $\mu$.

\begin{proposition} \label{propintersection}
Assume that $\lambda, \mu \vdash n$ with $Y^\lambda \succ Y^\mu$.
For any tableaux $T^\lambda, T^\mu$ there exist $i, j$ for which $T^\lambda(i,-) \cap T^\mu(-,j)$ 
contains at least two numbers.
Thus there exist two numbers in one row of $T^\lambda$ which appear in one column of $T^\mu$.
\end{proposition}

\begin{proof}
Write $\lambda = (n_1,\dots,n_k)$ and $\mu = (n'_1,\dots,n'_\ell)$.
We make the contrary assumption that $T^\lambda(i,-) \cap T^\mu(-,j)$ contains at most one number for all
$1 \le i \le k$ and $1 \le j \le n'_1$.
In particular, for $i = 1$ we see that the $n_1$ numbers in $T^\lambda(1,-)$ belong to different columns of $T^\mu$,
and so $n_1 \le n'_1$.
But $Y^\lambda \succ Y^\mu$ implies $n_1 \ge n'_1$, and so $n_1 = n'_1$.
The contrary assumption is not affected if we apply a vertical permutation to $T^\mu$,
and so there exists $v \in G_V(T^\mu)$ for which $T^\lambda(1,-) = (vT^\mu)(1,-)$ as sets;
these rows contain the same numbers, possibly in different order.

We now delete the first rows of $T^\lambda$ and $vT^\mu$, obtaining tableaux $T^{\lambda'} \succ T^{\mu'}$
where $\lambda', \mu'$ are partitions of $n-n_1$.
Both tableaux contain the numbers $\{ a_1, \dots, a_{n-n_1} \} \subset \{ 1, \dots, n \}$
which we can identify with $\{1 , \dots, n-n_1 \}$.
Repeating the argument of the first paragraph, we see that $n_2 = n'_2, \dots, n_k = n'_\ell$;
at the end we must have $k = \ell$.
This implies that $Y^\lambda = Y^\mu$, which is a contradiction.
\end{proof}

\begin{lemma} \label{lemmaiff}
Let $T$ be a tableau of shape $\lambda = (n_1,\dots,n_k) \vdash n$.
A permutation $p \in S_n$ has the form $p = hv$ for $h \in G_H(T)$ and $v \in G_V(T)$
if and only if $T(i,-) \cap (pT)(-,j)$ contains at most one number for all $i = 1, \dots, k$ and $j = 1, \dots, n_1$.
\end{lemma}

\begin{proof}
Assume that $p = hv$ for some $h \in G_H(T)$ and $v \in G_V(T)$.
Following Remark \ref{remarkconjugate}, we have $pT = hvT = (hvh^{-1}) hT$ where $hvh^{-1} \in G_V( hT )$.
If $x, y$ are distinct numbers in the same row of $T$, then they are in the same row but different columns of $hT$;
hence they are in different columns of $(hvh^{-1}) hT = pT$.

Conversely, assume that
$T(i,-) \cap (pT)(-,j)$ contains at most one number for all $i = 1, \dots, k$ and $j = 1, \dots, n_1$.
Then the numbers in the first column of $pT$ must appear in different rows of $T$.
We can apply a horizontal permutation $h_1 \in G_H(T)$ so that $(h_1 T)(-,1)$ is a permutation of $(pT)(-,1)$.
Similarly, the numbers in the second column of $pT$ must appear in different rows of $h_1 T$ and columns $j \ge 2$.
Keeping the numbers in $(h_1 T)(-,1)$ fixed, we can apply $h_2 \in G_H(T)$
so that $(h_2 h_1 T)(-,2)$ is a permutation of $(pT)(-,2)$.
Continuing, we obtain permutations $h_1, h_2, \dots, h_{n_1} \in G_H(T)$ so that
every number in $h T$ (where $h = h_{n_1} \cdots h_1$) is in the same column as in $pT$.
We now apply a vertical permutation $v' \in G_V( hT )$ to obtain $v' h T = p T$.
By Lemma \ref{lemmaconjugate}, we have $v' = hvh^{-1}$ for some $v \in G_V(T)$.
Therefore $pT = v'hT = hvh^{-1}hT = hvT$, as required.
\end{proof}

\begin{proposition}\label{proprowcolumn}
Assume that $\lambda = (n_1,\dots,n_k) \vdash n$, and let $T_1, \dots, T_{d_\lambda}$
be the standard tableaux of shape $\lambda$ in lex order.
If $r > s$ then there exist $i \in \{1,\dots,k\}$ and $j \in \{1,\dots,n_1\}$ such that $T_r(-,j) \cap T_s(i,-)$
contains at least two elements.
\end{proposition}

\begin{proof}
Let $(i',j')$ be the first position in which $T_r$ and $T_s$ have a different number.
Let $x, y$ be the numbers in position $(i',j')$ in $T_r, T_s$ respectively.
Since $r > s$ we have $x > y$.
In a standard tableau, each number in the first column is the least number that has not appeared in previous rows.
Hence $j' \ge 2$.
Suppose that $y$ occurs in position $(i'',j'')$ in $T_r$.
Since $T_r$ and $T_s$ are equal up to position $(i',j')$, we have two cases:
either $i'' = i'$ and $j'' > j'$ ($y$ is in the same row as $x$ but to the right),
or $i'' > i'$ ($y$ is in a lower row than $x$).
Since $x > y$ and $T_r$ is standard, the first case is impossible.
In the second case, $x > y$ implies $j'' < j'$ ($y$ must be in a column to the left of $x$).
We illustrate this situation with the diagram of Figure \ref{figure2}.
Since position $(i',j'')$ occurs before $(i',j')$, the number $z$ in this position must be the same
in both $T_r$ and $T_s$.
Hence $y, z$ are the two numbers in the same column of $T_r$ and the same row of $T_s$.
\end{proof}

\begin{figure}
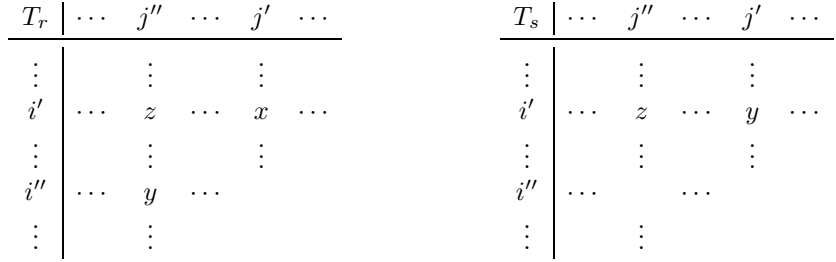

  \[
  \begin{array}{c|ccccc}
  T_r & \cdots & j'' & \cdots & j' & \cdots \\ \midrule
  \vdots & & \vdots & & \vdots & \\
  i' & \cdots & z & \cdots & x & \cdots \\
  \vdots & & \vdots & & \vdots \\
  i'' & \cdots & y & \cdots & & \\
  \vdots & & \vdots & & \\
  \end{array}
  \qquad\qquad\qquad
  \begin{array}{c|ccccc}
  T_s & \cdots & j'' & \cdots & j' & \cdots \\ \midrule
  \vdots & & \vdots & & \vdots & \\
  i' & \cdots & z & \cdots & y & \cdots \\
  \vdots & & \vdots & & \vdots \\
  i'' & \cdots & & \cdots & & \\
  \vdots & & \vdots \\
  \end{array}
  \]
\caption{Diagram for the proof of Proposition \ref{proprowcolumn}}
\label{figure2}
\end{figure}

%%%%%%%%%%%%%%%%%%%%%%%%%%%%%%%%%%%%%%%%%%%%%%%%%%%%%%%%%%%%%%%%%%%%%%%%%%%%%%%%%%%%%%%%%%%%%%%%%%%%%%%%%%%%

\subsection{Symmetric and alternating sums}

We construct special elements of $\mathbb{F} S_n$ which will be used to define idempotents in the group algebra.

\begin{definition}
Given a tableau $T$ of shape $\lambda \vdash n$ we define the following elements of $\mathbb{F} S_n$, 
where $\sgn\colon S_n \to \{ \pm 1 \}$ is the sign homomorphism:
  \[
  H_T = \sum_{h \in G_H(T)} h,
  \qquad \qquad
  V_T = \sum_{v \in G_V(T)} \sgn(v)\, v.
  \]
(Classically these were called the ``positive and negative symmetric groups'' for $T$.)
\end{definition}

\begin{lemma} \label{lemmasigns}
If $T$ is a tableau of shape $\lambda \vdash n$, and $h \in G_H(T)$, $v \in G_V(T)$, then
  \[
  h H_T = H_T = H_T h,
  \qquad
  v V_T = \sgn(v) V_T = V_T v.
  \]
\end{lemma}

\begin{proof}
For a horizontal permutation $h$, the function $G_H(T) \to G_H(T)$ sending $h' \mapsto hh'$ is a bijection,
and similarly for $h' \mapsto h'h$; this proves the claim for $H_T$.
Analogous bijections hold for $G_V(T)$ and a vertical permutation $v$, so
  \begin{align*}
  v V_T
  &=
  \sum_{v' \in G_V(T)} \!\!\!\! \sgn(v')\, vv'
  =
  \sgn(v)^{-1} \!\!\!\! 
  \sum_{v' \in G_V(T)} \!\!\!\! \sgn(v)\sgn(v')\, vv'
  =
  \sgn(v) \!\!\!\! 
  \sum_{v' \in G_V(T)} \!\!\!\! \sgn(vv')\, vv'
  =
  \sgn(v) V_T.
  \end{align*}
The proof that $\sgn(v) V_T = V_T v$ is similar.
\end{proof}

\begin{proposition} \label{propositionconjugate}
If $T$ is a tableau of shape $\lambda \vdash n$, and $p \in S_n$, then
  \[
  H_{pT} = p \, H_T \, p^{-1},
  \qquad
  V_{pT} = p \, V_T \, p^{-1}.
  \]
\end{proposition}

\begin{proof}
This follows from Lemmas \ref{lemmaconjugate} and \ref{lemmasigns} since $\sgn(p) = \sgn(p^{-1})$.
\end{proof}

%%%%%%%%%%%%%%%%%%%%%%%%%%%%%%%%%%%%%%%%%%%%%%%%%%%%%%%%%%%%%%%%%%%%%%%%%%%%%%%%%%%%%%%%%%%%%%%%%%%%%%%%%%%%

\subsection{Idempotents and orthogonality in the group algebra}

We construct idempotent elements in $\mathbb{F} S_n$ and study their orthogonality properties.

\begin{definition}
Let $T_1, \dots, T_{n!}$ be the tableaux of shape $\lambda \vdash n$ in lex order.
For $1 \le i, j \le n!$ we define $s_{ij} \in S_n$ by $s_{ij} T_j = T_i$;
clearly $s_{ji} = s_{ij}^{-1}$ and $s_{ij} s_{jk} = s_{ik}$.
We also define these group algebra elements (omitting $\lambda$ if it is understood):
  \[
  D^\lambda_i = H_{T_i} V_{T_i} = \sum_{h \in G_H(T_i)} \sum_{v \in G_V(T_i)} \sgn(v) \, hv.
  \]
\end{definition}

\begin{proposition} \label{propconj}
If $T$ is a tableau of shape $\lambda \vdash n$ then
  \begin{equation}
  \label{d=sds}
  D_j = s_{ji} D_i s_{ij}.
  \end{equation}
Equivalently,
  \begin{equation}
  \label{sd=ds}
  s_{ij} D_j = D_i s_{ij}.
  \end{equation}
\end{proposition}

\begin{proof}
Using Proposition \ref{propositionconjugate} and the definition of $s_{ij}$, we obtain
\[
s_{ji} D_i s_{ij}
=
s_{ji} H_{T_i} V_{T_i} s_{ji}^{-1}
=
\big[ s_{ji} H_{T_i} s_{ji}^{-1} \big] \! \big[ s_{ji} V_{T_i} s_{ji}^{-1} \big]
=
H_{ s_{ji} T_i} V_{s_{ji} T_i}
=
H_{T_j} V_{T_j}
=
D_j.
\]
This proves the first equation, and the second follows from $s_{ji} = s_{ij}^{-1}$.
\end{proof}

\begin{proposition} \label{prop0product}
If $\lambda, \mu \vdash n$ with $\lambda \ne \mu$, then $D_i^\lambda D_j^\mu = 0$ for all tableaux 
$T_i^\lambda, T_j^\mu$.
\end{proposition}

\begin{proof}
We first assume $Y^\lambda \prec Y^\mu$.
Proposition \ref{propintersection} shows that there exist two numbers $k, \ell$ in the same row of $T^\mu$ and 
the same column of $T^\lambda$; we now omit the subscripts $i,j$.
For the transposition $t  = (k, \ell)$ we have $t \in G_V(T^\lambda)$ and $t \in G_H(T^\mu)$.
Using Lemma \ref{lemmasigns} we obtain
  \begin{align*}
  D^\lambda D^\mu
  &=
  H_{T^\lambda} V_{T^\lambda} H_{T^\mu} V_{T^\mu}
  =
  H_{T^\lambda} V_{T^\lambda} \, t^2 H_{T^\mu} V_{T^\mu}
  =
  H_{T^\lambda} ( V_{T^\lambda} t ) ( t H_{T^\mu} ) V_{T^\mu}
  \\
  &=
  H_{T^\lambda} ( - V_{T^\lambda} ) ( H_{T^\mu} ) V_{T^\mu}
  =
  - H_{T^\lambda} V_{T^\lambda} H_{T^\mu} V_{T^\mu}
  =
  - D^\lambda D^\mu.
  \end{align*}
Hence $D^\lambda D^\mu = 0$.
On the other hand, if $Y^\lambda \succ Y^\mu$, then Proposition \ref{propositionconjugate} implies that
for any $p \in S_n$ we have
  \[
  H_{T^\lambda} p V_{T^\mu}
  =
  H_{T^\lambda} \big( p V_{T^\mu} p^{-1} \big) p
  =
  H_{T^\lambda} V_{p T^\mu} p.
  \]
Proposition \ref{propintersection} shows that there exist two numbers $k, \ell$ in the same row of $T^\lambda$ 
and the same column of $pT^\mu$.
Then $t = (k,\ell) \in G_V(pT^\mu) \cap G_H(T^\lambda)$ and so
  \[
  H_{T^\lambda} V_{p T^\mu} p
  =
  H_{T^\lambda} t^2 V_{p T^\mu} p
  =
  ( H_{T^\lambda} t ) ( t V_{p T^\mu} ) p
  =
  ( H_{T^\lambda} ) ( - V_{p T^\mu} ) p
  =
  - H_{T^\lambda} V_{p T^\mu} p.
  \]
Hence $H_{T^\lambda} V_{p T^\mu} p = 0$ and so $H_{T^\lambda} p V_{T^\mu} = 0$, for all $p \in S_n$.
Therefore
  \[
  D^\lambda D^\mu
  =
  H_{T^\lambda} V_{T^\lambda} H_{T^\mu} V_{T^\mu}
  =
  H_{T^\lambda} \bigg( \sum_{p \in S_n} x_p \, p \bigg) V_{T^\mu}
  =
  \sum_{p \in S_n} x_p \big( H_{T^\lambda} p V_{T^\mu} \big)
  =
  0,
  \]
where $x_p \in \mathbb{F}$ for all $p \in S_n$.
This completes the proof.
\end{proof}

\begin{corollary} \label{corollary0}
Let $\lambda \vdash n$, and let $T_1, \dots, T_{d_\lambda}$ be the standard tableaux in lex order.
If $i > j$ then $D_i D_j = 0$.
\end{corollary}

\begin{proof}
We write $H_i, V_i$ for $H_{T_i}, V_{T_i}$.
By Proposition \ref{proprowcolumn}, there exist two numbers $k,\ell$ in the same column of $T_i$ 
and the same row of $T_j$.
Using the transposition $t = (k,\ell)$ and Lemma \ref{lemmasigns} we obtain
  \[
  D_i D_j
  =
  H_i V_i H_j V_j
  =
  H_i V_i t^2 H_j V_j
  =
  H_i ( V_i t ) ( t H_j ) V_j
  =
  H_i ( - V_i ) ( H_j ) V_j
  =
  - D_i D_j.
  \]
Therefore $D_i D_j = 0$.
\end{proof}

\begin{proposition}
\emph{\textbf{(von Neumann's Theorem)}}
Let $\lambda \vdash n$.
For $i = 1, \dots, n!$ we have $D_i^2 = c_i D_i$ where $c_i = n!/f_i$, and $f_i$ is the dimension of
the left ideal $\mathbb{F} S_n D_i$.
\end{proposition}

\begin{proof}
For scalars $x_p \in \mathbb{F}$ which we will determine, we write
  \[
  D_i^2 = \sum_{p \in S_n} x_p \, p.
  \] 
For any $h \in G_H(T_i)$ and $v \in G_V(T_i)$ we have
  \begin{align*}
  h D_i^2 v
  &=
  h \bigg( \sum_{p \in S_n} x_p \, p \bigg) v
  =
  \sum_{p \in S_n} x_p \, hpv,
  \\
  h D_i^2 v
  &=
  ( h H_i ) V_i H_i ( V_i v )
  =
  \sgn(v) H_i V_i H_i V_i
  =
  \sgn(v) D_i^2.
  \end{align*}
Therefore
  \begin{equation}
  \label{coeffs}
  \sum_{p \in S_n} x_p \, hpv = \sgn(v) \sum_{p \in S_n} x_p \, p.
  \end{equation}
Each permutation in $S_n$ occurs once and only once on each side of this equation.

First, consider the coefficient in $D_i^2$ of a permutation of the form $hv$.
On the left side of \eqref{coeffs} take $p = \iota$, on the right side take $p = hv$, and compare coefficients:
  \[
  x_\iota = \sgn(v) x_{hv}.
  \]
Hence $x_{hv} = \sgn(v) x_\iota$.
Second, consider the coefficient in $D_i^2$ of a permutation $q$ not of the form $hv$.
Lemma \ref{lemmaiff} implies that there are two numbers $k,\ell$ in the same row of $T_i$ and the same column of $qT_i$.
For the transposition $t = (k,\ell)$, we have $t \in G_H(T_i)$ and $q^{-1}tq \in G_V(T_i)$.
We can take $h = t$ and $v = q^{-1}tq$ in equation \eqref{coeffs}:
  \[
  \sum_{p \in S_n} x_p \, tpq^{-1}tq = \sgn(q^{-1}tq) \sum_{p \in S_n} x_p \, p.
  \]
Setting $p = q$ on both sides, we obtain
  \[
  x_q \, tqq^{-1}tq = \sgn(q^{-1}tq) x_q \, q,
  \]
and this simplifies to $x_q q = -x_q q$, implying $x_q = 0$.
Combining the results of the two cases, we obtain $D_i^2 = c_i D_i$ where $c_i = x_\iota$.

It remains to show that $x_\iota = n!/f_i$.
We choose a basis for the left ideal $\mathbb{F} S_n D_i$ consisting of elements $p_1 D_i, \dots, p_{f_i} D_i$ where
$p_1, \dots, p_{f_i} \in S_n$, and extend this to a basis of $\mathbb{F} S_n$.
We regard $D_i$ as a linear operator on $\mathbb{F} S_n$, acting by right multiplication.
The matrix representing $D_i$ with respect to our basis has the form
  \[
  \left[
  \begin{array}{cc}
  x_\iota I_{f_i} & \ast \\
  0 & 0
  \end{array}
  \right],
  \]
where $\ast$ indicates irrelevant entries.
Hence $\mathrm{trace}( D_i ) = x_\iota f_i$.
On the other hand, since $\mathrm{trace}(q) = 0$ for $q \ne \iota$, we have
  \[
  \mathrm{trace}( D_i )
  =
  \mathrm{trace}\bigg( \sum_{h,v} \sgn(v) hv \bigg)
  =
  \sum_{h,v} \sgn(v) \mathrm{trace}( hv )
  =
  \mathrm{trace}\big( I_{\mathbb{F}S_n} \big)
  =
  n!.
  \]
Now we have $x_\iota f_i = n!$, so $c_i = x_\iota = n!/f_i$.
\end{proof}

\begin{definition} \label{defscalar}
Let $T^\lambda_1, \dots, T^\lambda_{n!}$ be all the tableaux of shape $\lambda \vdash n$. 
We define
  \[
  E^\lambda_i = \frac{f_i}{n!} D^\lambda_i \qquad (i = 1, \dots, n!).
  \]
\end{definition}

\begin{corollary}
Every $E^\lambda_i$ is an idempotent: $( E^\lambda_i )^2 = E^\lambda_i$.
\end{corollary}

%%%%%%%%%%%%%%%%%%%%%%%%%%%%%%%%%%%%%%%%%%%%%%%%%%%%%%%%%%%%%%%%%%%%%%%%%%%%%%%%%%%%%%%%%%%%%%%%%%%%%%%%%%%%

\subsection{Two-sided ideals in the group algebra}

The results in this subsection lead us towards an explicit description of the isomorphism $\psi$ in the
Wedderburn decomposition \eqref{iso}.

\begin{definition}
If $T_i$, $T_j$ are tableaux of shape $\lambda \vdash n$ then we define
  \[
  \esgn_{ij} =
  \begin{cases}
  \, \sgn(v) &\text{if $s_{ji} = vh$ for $h \in G_H(T_i)$ and $v \in G_V(T_i)$}
  \\
  \, 0 &\text{otherwise}
  \end{cases}
  \]
\end{definition}

\begin{lemma} \label{lemmaeiej}
If $T_i$, $T_j$ are tableaux of shape $\lambda \vdash n$ then $E_i E_j = \esgn_{ij} E_i s_{ij}$.
\end{lemma}

\begin{proof}
First, assume that $s_{ji} = vh$ for some $h \in G_H(T_i)$, $v \in G_V(T_i)$.
Proposition \ref{propconj}, equation \eqref{d=sds} and von Neumann's Theorem imply
  \begin{align*}
  E_i E_j
  &=
  E_i ( s_{ji} E_i s_{ij} )
  =
  \frac1{c_i^2} H_i ( V_i v ) ( h H_i ) V_i s_{ij}
  =
  \frac1{c_i^2} \sgn(v) H_i V_i H_i V_i s_{ij}
  \\
  &=
  \sgn(v) E_i^2 s_{ij}
  =
  \sgn(v) E_i s_{ij}.
  \end{align*}
Second, assume that $s_{ji} \ne vh$ for any $h \in G_H(T_i)$, $v \in G_V(T_i)$.
Since $T_j = s_{ji} T_i$, Lemma \ref{lemmaiff} shows that there are numbers $k,\ell$ in the same column of $T_i$
and the same row of $T_j$.
Using the transposition $t = (k,\ell) \in G_V(T_i) \cap G_H(T_j)$ we obtain
  \[
  E_i E_j
  =
  \frac1{c_i^2} H_i ( V_i t ) ( t H_j ) V_j
  =
  - \frac1{c_i^2} H_i V_i H_j V_j
  =
  - E_i E_j.
  \]
Hence $E_i E_j = 0$.
\end{proof}

\begin{remark}
From now on we will work only with standard tableaux.
\end{remark}

\begin{definition} \label{defbigeps}
Given a partition $\lambda \vdash n$ with standard tableaux $T_1, \dots, T_{d_\lambda}$ in lex order,
we write $\Epsilon^\lambda$ for the $d_\lambda \times d_\lambda$ matrix with $(i,j)$ entry $\esgn_{ij}$.
\end{definition}

\begin{lemma} \label{E=I+N.lemma}
We have $\Epsilon^\lambda = I^\lambda + \mathcal{F}^\lambda$ where
$I^\lambda$ is the identity matrix and $\mathcal{F}^\lambda$ is a strictly upper triangular matrix.
In particular, $\Epsilon^\lambda$ is invertible.
\end{lemma}

\begin{proof}
If $i > j$ then Corollary \ref{corollary0} implies that $E_i E_j = 0$ and so $\esgn_{ij} = 0$.
If $i = j$ then $s_{ii} = \iota$ and so Lemma \ref{lemmaeiej} gives $E_i = \esgn_{ii} E_i$, 
hence $\esgn_{ii} = 1$.
\end{proof}

\begin{proposition} \label{almostmatrixunits}
If $\lambda \vdash n$ and $T_i, T_j, T_k, T_\ell$ are standard tableaux of shape $\lambda$ then
  \[
  ( E_i s_{ij} )( E_k s_{k\ell} ) = \esgn_{jk} E_i s_{i\ell}.
  \]
\end{proposition}

\begin{proof}
Using Proposition \ref{propconj}, equation \eqref{sd=ds}, and Lemma \ref{lemmaeiej}, we obtain
  \[
  E_i s_{ij} E_k s_{k\ell}
  =
  s_{ij} E_j E_k s_{k\ell}
  =
  \esgn_{jk} s_{ij} E_j s_{jk} s_{k\ell}
  =
  \esgn_{jk} E_i s_{ij} s_{jk} s_{k\ell}
  =
  \esgn_{jk} E_i s_{i\ell},
  \]
as required.
\end{proof}

\begin{remark}
If we replace the scalar $\esgn_{jk}$ in Proposition \ref{almostmatrixunits} by the Kronecker delta 
$\delta_{jk}$, and write $E_{ij} = E_i s_{ij}$, then we obtain the matrix unit relations 
$E_{ij} E_{k\ell} = \delta_{jk} E_{i\ell}$.
In order to construct the isomorphism $\psi$, we need to modify the elements $E_i s_{ij}$ to produce other
elements which exactly satisfy the matrix unit relations. 
\end{remark}

\begin{definition}
We write $N^\lambda$ for the the subspace spanned by the $E^\lambda_i s^\lambda_{ij}$:
  \[
  N^\lambda = \mathrm{span} \{ E^\lambda_i s^\lambda_{ij} \mid 1 \le i, j \le d_\lambda \} \subset \mathbb{F} S_n.
  \]
We write $N$ for the sum of the subspaces $N^\lambda$ over all $\lambda \vdash n$.
\end{definition}

\begin{corollary}
For each $\lambda \vdash n$, the subspace $N^\lambda$ is a subalgebra of $\mathbb{F} S_n$.
\end{corollary}

We fix a partition $\lambda \vdash n$ with standard tableaux $T_1, \dots, T_{d_\lambda}$ in lex order.
Let $A = ( a_{ij} )$ be any $d_\lambda \times d_\lambda$ matrix over $\mathbb{F}$, and consider
the group algebra element
  \begin{equation}
  \label{defalpha}
  \alpha^\lambda(A) = \sum_{i=1}^{d_\lambda} \sum_{j=1}^{d_\lambda} a_{ij} E_i s_{ij}.
  \end{equation}
As usual, we write $E_{ij}$ for the $d_\lambda \times d_\lambda$ matrix with 1 in position $(i,j)$ and 0 elsewhere.

\begin{lemma} \label{lemmaE}
For all partitions $\lambda \vdash n$ and all $i,j,k,\ell \in \{1,\dots,d_\lambda\}$ we have
  \[
  \alpha^\lambda(E_{ij}) \alpha^\lambda(E_{k\ell}) = \alpha^\lambda( E_{ij} \Epsilon^\lambda E_{k\ell} ).
  \]
\end{lemma}

\begin{proof}
We have
  $
  \alpha^\lambda(E_{ij}) \alpha^\lambda(E_{k\ell})
  =
  E_i s_{ij} E_k s_{k\ell}
  =
  \esgn_{jk} E_i s_{i\ell}
  =
  \alpha^\lambda( E_{ij} \Epsilon^\lambda E_{k\ell} )
  $,
using Proposition \ref{almostmatrixunits}.
\end{proof}

\begin{proposition} \label{proplinind}
The set $\{ E^\mu_i s^\mu_{ij} \mid \mu \vdash n, 1 \le i, j \le d_\mu \}$ is linearly independent.
\end{proposition}

\begin{proof}
A linear dependence relation among the $E^\mu_i s^\mu_{ij}$ can be written as
  \[
  \sum_{\mu \vdash n} \alpha^\mu( A^\mu ) = 0.
  \]
We fix a partition $\lambda$, and obtain
  \[
  \alpha^\lambda \big( E_{ii} ( \Epsilon^\lambda )^{-1} \big)
  \bigg[
  \sum_{\mu \vdash n} \alpha^\mu( A^\mu )
  \bigg]
  \alpha^\lambda \big( ( \Epsilon^\lambda )^{-1} E_{jj} \big)
  = 0.
  \]
Using equation \eqref{sd=ds}, Definition \ref{defscalar}, and Proposition \ref{prop0product},
we see that all terms vanish except for $\mu = \lambda$:
  \[
  \alpha^\lambda \big( E_{ii} ( \Epsilon^\lambda )^{-1} \big)
  \alpha^\lambda( A^\lambda )
  \alpha^\lambda \big( ( \Epsilon^\lambda )^{-1} E_{jj} \big)
  = 0.
  \]
Lemma \ref{lemmaE} gives
  $\alpha^\lambda (
  E_{ii} ( \Epsilon^\lambda )^{-1}
  \Epsilon^\lambda
  A^\lambda
  \Epsilon^\lambda
  ( \Epsilon^\lambda )^{-1} E_{jj}
  )
  = 0$,
hence $\alpha^\lambda ( E_{ii} A^\lambda E_{jj} ) = 0$ \and $\alpha^\lambda ( a^\lambda_{ij} E_{ij} ) = 0$,
and so $a^\lambda_{ij} E_i s_{ij} = 0$.
Thus $a^\lambda_{ij} = 0$ for all $\lambda$ and all $i, j$.
\end{proof}

\begin{definition}
Suppose that $n$ has $r$ distinct partitions $\lambda_1, \dots, \lambda_r$ in lex order.
For $i = 1, \dots, r$ let $d_i = d_{\lambda_i}$ be the number of standard tableaux of shape $\lambda_i$.
Consider the direct sum of full matrix algebras
  \[
  M = \bigoplus_{i=1}^r M_{d_i}(\mathbb{F}).
  \]
The linear map $\alpha\colon M \to \mathbb{F} S_n$ is the direct sum of the $\alpha^i = \alpha^{\lambda_i}$
from equation \eqref{defalpha}:
  \[
  \alpha( A_1, \dots, A_r ) = \alpha^1(A_1) + \cdots + \alpha^r(A_r).
  \]
\end{definition}

\begin{corollary}
\label{cordl=fi}
The map $\alpha$ is injective.
For every $\lambda \vdash n$ and $1 \le i, j \le d_\lambda$, we have $\dim N^\lambda = d_\lambda^{\,2}$.
The sum $N$ of the $N^\lambda$ is direct, and hence $\dim N = \sum_\lambda d_\lambda^2$.
\end{corollary}

\begin{proof}
Injectivity of $\alpha$ is equivalent to the linear independence in Proposition \ref{proplinind}.
Since linear independence holds for each $\lambda$, the spanning set for $N^\lambda$ is also a basis.
The sum of the $N^\lambda$ is direct by Proposition \ref{prop0product}.
\end{proof}

Since $N \subseteq \mathbb{F} S_n$, it follows that $\sum_\lambda d_\lambda^2 \le n!$,
so to prove $N = \mathbb{F} S_n$, it remains to show equality.
Algorithms for insertion or deletion of a number to or from a standard tableau provide
a bijection between $S_n$ and the set of ordered pairs of standard tableaux of the same shape.
For details, see \cite[\S 5.1.4, Theorem A]{Knuth1998}.

%%%%%%%%%%%%%%%%%%%%%%%%%%%%%%%%%%%%%%%%%%%%%%%%%%%%%%%%%%%%%%%%%%%%%%%%%%%%%%%%%%%%%%%%%%%%%%%%%%%%%%%%%%%%

\subsection{Matrix units in the group algebra} \label{matrixunits}

We prove that the map $\psi$ in \eqref{iso} is an isomorphism by constructing elements of $\mathbb{F}S_n$ 
corresponding to matrix units.

\begin{remark}
The linear map $\alpha^\lambda\colon M_{d_\lambda}(\mathbb{F}) \to \mathbb{F} S_n$ is not in general 
an algebra homomorphism.
However, we can easily obtain an algebra homomorphism from it.
\end{remark}

\begin{definition} \label{defU}
For all $\lambda \vdash n$ and $1 \le i, j \le d_\lambda$, we define the following elements:
  \[
  U^\lambda_{ij} = \alpha^\lambda \big( E^\lambda_{ij} (\mathcal{E}^\lambda)^{-1} \big) \in \mathbb{F} S_n.
  \]
\end{definition}

\begin{proposition}
For all $\lambda, \mu \vdash n$, $1 \le i, j \le d_\lambda$, $1 \le k, \ell \le d_\mu$ we have
  \[
  U^\lambda_{ij} U^\mu_{k\ell} = \delta_{\lambda\mu} \delta_{jk} U^\lambda_{i\ell}.
  \]
\end{proposition}

\begin{proof}
If $\lambda = \mu$ then
  \begin{align*}
  U_{ij} U_{k\ell}
  &=
  \alpha( E_{ij} \mathcal{E}^{-1} ) \alpha( E_{k\ell} \mathcal{E}^{-1} )
  =
  \alpha( E_{ij} \mathcal{E}^{-1} \mathcal{E} E_{k\ell} \mathcal{E}^{-1} )
  =
  \alpha( E_{ij} E_{k\ell} \mathcal{E}^{-1} )
  \\
  &=
  \alpha( \delta_{jk} E_{i\ell} \mathcal{E}^{-1} )
  =
  \delta_{jk} \alpha( E_{i\ell} \mathcal{E}^{-1} )
  =
  \delta_{jk} U_{i\ell}.
  \end{align*}
The factor $\delta_{\lambda\mu}$ comes from the orthogonality of Proposition \ref{prop0product}.
\end{proof}

\begin{definition} \label{defpsi}
We define the linear map $\psi\colon M \to \mathbb{F} S_n$ on matrix units as
  \[
  \psi( E^\lambda_{ij} ) = U^\lambda_{ij} \qquad (\lambda \vdash n; \, 1 \le i, j \le d_\lambda).
  \]
\end{definition}

\begin{theorem} \label{phi-inverse}
The map $\psi\colon M \to \mathbb{F} S_n$ is an isomorphism of associative algebras.
In particular, $M_{d_i}(\mathbb{F})$ is isomorphic to $N^{\lambda_i}$.
\end{theorem}

\begin{proof}
This is an immediate corollary of the preceding results.
\end{proof}

\begin{remark} \label{charpremark}
Since the direct sum $M$ of full matrix algebras is clearly semisimple, and simplicity is
preserved by isomorphism, it follows that $\mathbb{F} S_n$ is semisimple, and moreover that it
splits over $\mathbb{F}$: the structure theory
of semisimple associative algebras implies that $\mathbb{F} S_n$ is isomorphic to the direct sum of simple
two-sided ideals, and that each simple ideal is isomorphic to the endomorphism algebra of a vector space over
a division ring $\mathbb{D}$ over $\mathbb{F}$.
But our results show that $\mathbb{D} = \mathbb{F}$ for every $\lambda$.
Since the scalar factors $d_\lambda/n!$ in Definition \ref{defscalar} are defined in characteristic $> n$,
we also obtain the semisimplicity of $\mathbb{F} S_n$ in this case.
\end{remark}

\begin{example} \label{example.n=3}
For $n = 3$ we take the permutations 123, 132, 213, 231, 312, 321 
in lex order -- writing $p$ as $p(1)p(2)p(3)$ -- as our basis of $\mathbb{F} S_3$.
The partitions $\lambda = 3$, $\mu = 21$, $\nu = 111$ have the following standard tableaux:
  \[
  T^\lambda_1 = \begin{array}{cc} \young(123) \end{array} \qquad\quad
  T^\mu_1 =\begin{array}{cc} \young(12,3) \end{array} \qquad\quad
  T^\mu_2 = \begin{array}{cc} \young(13,2) \end{array} \qquad\quad
  T^\nu_1 = \begin{array}{cc} \young(1,2,3) \end{array}
  \]
Thus $d_\lambda = 1$, $d_\mu = 2$, $d_\nu = 1$ and hence we have the isomorphism
  \[
  \psi\colon M = \mathbb{F} \oplus M_2(\mathbb{F}) \oplus \mathbb{F} \longrightarrow \mathbb{F} S_3.
  \]
As ordered basis of $M$ we take the matrix units 
$E^\lambda_{11}$, $E^\mu_{11}$, $E^\mu_{12}$, $E^\mu_{21}$, $E^\mu_{22}$, $E^\nu_{11}$.
We will compute the corresponding elements $U^\rho_{ij}$ of $\mathbb{F}S_n$.
The groups of horizontal and vertical permutations are as follows:
  \begin{alignat*}{2}
  &G_H(T^\lambda_1) = S_3 &\qquad\qquad &G_V(T^\lambda_1) = \{ 123 \}
  \\
  &G_H(T^\mu_1) = \{ 123, 213 \} &\qquad\qquad &G_V(T^\mu_1) = \{ 123, 321 \}
  \\
  &G_H(T^\mu_2) = \{ 123, 321 \} &\qquad\qquad &G_V(T^\mu_2) = \{ 123, 213 \}
  \\
  &G_H(T^\nu_1) = \{ 123 \} &\qquad\qquad &G_V(T^\nu_1) = S_3.
  \end{alignat*}
The symmetric and alternating sums over these subgroups are as follows:
  \[
  \begin{array}{lcr}
  \multicolumn{2}{l}{H_{T^\lambda_1} = 123 + 132 + 213 + 231 + 312 + 321} & V_{T^\lambda_1} = 123
  \\[3pt]
  H_{T^\mu_1} = 123 + 213 & & V_{T^\mu_1} = 123 - 321
  \\[3pt]
  H_{T^\mu_2} = 123 + 321 & & V_{T^\mu_2} = 123 - 213
  \\[3pt]
  H_{T^\nu_1} = 123 &\multicolumn{2}{r}{V_{T^\nu_1} = 123 - 132 - 213 + 231 + 312 - 321}
  \end{array}
  \]
The products $D^\rho_{ij}$ are easily calculated; and scaling gives the idempotents:
  \begin{align*}
  E^\lambda_1 &= \tfrac16 ( 123 + 132 + 213 + 231 + 312 + 321 )
  \\
  E^\mu_1 &= \tfrac13 ( 123 + 213 - 312 - 321 ),
  \qquad
  E^\mu_2 = \tfrac13 ( 123 - 213 - 231 + 321 )
  \\
  E^\nu_1 &= \tfrac16 ( 123 - 132 - 213 + 231 + 312 - 321 )
  \end{align*}
Clearly $s^\mu_{12} = s^\mu_{21} = 132$, and this is the only non-trivial case. 
Hence $s_{12} \ne vh$ for any $v \in G_V(T^\mu_2)$, $h \in G_H(T^\mu_2)$ (see Lemma \ref{lemmaeiej}),
and so every $\Epsilon^\rho$ is the identity matrix of size $d_\rho$.
Therefore every $U^\rho_{ij} = \alpha^\rho( E_{ij} ) = E^\rho_i s^\rho_{ij}$, 
which gives the following matrix units in the group algebra:
  \begin{align*}
  U^\lambda_{11} &= E^\lambda_1
  \\
  U^\mu_{11} &= E^\mu_1,
  \qquad
  U^\mu_{12} = E^\mu_1 s_{12} = \tfrac13 ( 132 + 231 - 312 - 321 )
  \\
  U^\mu_{21} &= E^\mu_2 s_{21} = \tfrac13 ( 132 - 213 - 231 + 312 ),
  \qquad
  U^\mu_{22} = E^\mu_2
  \\
  U^\nu_{11} &= E^\nu_1
  \end{align*}
These equations can be summarized by the matrix representing $\psi$ with respect to our ordered bases of 
$M$ and $\mathbb{F} S_n$, and then we obtain the matrix representing $\psi^{-1}$:
  \[
  \psi
  \sim
  \frac16 \!
  \left[
  \begin{array}{rrrrrr}
  1 &\!  2 &\!  0 &\!  0 &\!  2 &\!  1 \\
  1 &\!  0 &\!  2 &\!  2 &\!  0 &\! -1 \\
  1 &\!  2 &\!  0 &\! -2 &\! -2 &\! -1 \\
  1 &\!  0 &\!  2 &\! -2 &\! -2 &\!  1 \\
  1 &\! -2 &\! -2 &\!  2 &\!  0 &\!  1 \\
  1 &\! -2 &\! -2 &\!  0 &\!  2 &\! -1
  \end{array}
  \right]
  \qquad
  \psi^{-1}
  \sim
  \!
  \left[
  \begin{array}{rrrrrr}
  1 &\!  1 &\!  1 &\!  1 &\!  1 &\!  1 \\
  1 &\!  0 &\!  1 &\! -1 &\!  0 &\! -1 \\
  0 &\!  1 &\! -1 &\!  1 &\! -1 &\!  0 \\
  0 &\!  1 &\!  0 &\! -1 &\!  1 &\! -1 \\
  1 &\!  0 &\! -1 &\!  0 &\! -1 &\!  1 \\
  1 &\! -1 &\! -1 &\!  1 &\!  1 &\! -1
  \end{array}
  \right]
  \]
For any $X \in \mathbb{F} S_3$, we have
  \[
  \psi^{-1}( X )
  =
  x_1 E^\lambda_{11} + x_2 E^\mu_{11} + x_3 E^\mu_{12} + x_4 E^\mu_{21} + x_5 E^\mu_{22} + x_6 E^\nu_{11}
  =
  \left[ \, x_1, \begin{bmatrix} x_2 &\!\! x_3 \\ x_4 &\!\! x_5 \end{bmatrix}\!\!, \, x_6 \, \right]
  \]
and therefore  
  \begin{alignat*}{2}
  \psi^{-1}( 123 ) &= \left[ \, 1, \begin{bmatrix}  1 &  0 \\  0 &  1 \end{bmatrix}\!,  1 \, \right]
  &\qquad\qquad
  \psi^{-1}( 132 ) &= \left[ \, 1, \begin{bmatrix}  0 &  1 \\  1 &  0 \end{bmatrix}\!, -1 \, \right]
  \\
  \psi^{-1}( 213 ) &= \left[ \, 1, \begin{bmatrix}  1 & -1 \\  0 & -1 \end{bmatrix}\!, -1 \, \right]
  &\qquad\qquad
  \psi^{-1}( 231 ) &= \left[ \, 1, \begin{bmatrix} -1 &  1 \\ -1 & 0 \end{bmatrix}\!,   1 \, \right]
  \\
  \psi^{-1}( 312 ) &= \left[ \, 1, \begin{bmatrix}  0 & -1 \\  1 & -1 \end{bmatrix}\!,  1 \, \right]
  &\qquad\qquad
  \psi^{-1}( 321 ) &= \left[ \, 1, \begin{bmatrix} -1 &  0 \\ -1 &  1 \end{bmatrix}\!, -1 \, \right]
  \end{alignat*}
These are the representation matrices for the irreducible representations of $S_3$.
\end{example}

%%%%%%%%%%%%%%%%%%%%%%%%%%%%%%%%%%%%%%%%%%%%%%%%%%%%%%%%%%%%%%%%%%%%%%%%%%%%%%%%%%%%%%%%%%%%%%%%%%%%%%%%%%%%

\subsection{Clifton's theorem on representation matrices}

Our next goal is to compute explicitly the algebra homomorphism $\phi$:
  \[
  \phi\colon \mathbb{F} S_n \longrightarrow \bigoplus_{\lambda \vdash n} M_{d_\lambda}(\mathbb{F})
  \]
We fix $\lambda \vdash n$ throughout the following discussion, and consider all the tableaux $T_1, \dots, T_{n!}$
of shape $\lambda$.
Recall that for $1 \le i, j \le n!$ we define $s_{ij} \in S_n$ by the equation $s_{ij} T_j = T_i$.
If $p \in S_n$ then $p T_j = T_r$ for some $r$, and so $p = s_{rj}$.
As before, we write $E_i$ for the idempotent corresponding to $T_i$.
Proposition \ref{propconj} and Lemma \ref{lemmaeiej} show that 
$E_i E_j = \esgn_{ij} E_i s_{ij} = \esgn_{ij} s_{ij} E_j$.
Therefore
  \[
  E_i p E_j
  =
  E_i s_{rj} E_j
  =
  E_i E_r s_{rj}
  =
  \esgn_{ir} s_{ir} E_r s_{rj}
  =
  \esgn_{ir} s_{ir} s_{rj} E_j
  =
  \esgn_{ir} s_{ij} E_j.
  \]
We define $\esgn_{ij}^p = \esgn_{ir}$ when $p = s_{rj}$, so that for all $i, j, p$ we have
  \begin{equation}
  \label{eipej}
  E_i p E_j = \esgn_{ij}^p s_{ij} E_j.
  \end{equation}
We now restrict to the $d_\lambda$ standard tableaux $T_1, \dots, T_{d_\lambda}$ in lex order.

\begin{definition}
For all $p \in S_n$ the \textbf{Clifton matrix} $A^\lambda_p$ is defined by
  \[
  ( A^\lambda_p )_{ij} = \esgn_{ij}^p \qquad (1 \le i, j \le d_\lambda).
  \]
The matrix previously denoted $\mathcal{E}^\lambda$ is the Clifton matrix $A^\lambda_\iota$ for $\iota \in S_n$.
\end{definition}

Referring to the definition of $\esgn_{ij}$ in Lemma \ref{lemmaeiej}, 
we see that $A^\lambda_p$ can be computed by the following steps, 
presented formally in Figure \ref{cliftonmatrixalgorithm}; 
see \cite{Clifton1982}, \cite{Bergdolt1995}, \cite[Figure 1]{BP2011}:
  \begin{itemize}
  \item
  Apply $p$ to the standard tableau $T_j$ obtaining the (possibly non-standard) tableau $p T_j$.
  \item
  If there exist two numbers that appear together both in a column of $T_i$ and in a row of $p T_j$,
  then $(A^\lambda_p)_{ij} = 0$.
  \item
  Otherwise, there exists a vertical permutation $q \in G_V(T_i)$ which takes the numbers of $T_i$ into 
  the rows they occupy in $p T_j$.
  Then $(A^\lambda_p)_{ij} = \epsilon( q )$.
  \end{itemize}
Figure \ref{cliftonmatrixalgorithm} attempts to find $q$, and returns 0 if no such permutation exists.

Before proving Clifton's theorem, it is worth quoting in its entirety the review in MathSciNet (MR0624907)
by G. D. James of Clifton's paper \cite{Clifton1982}:
``From his natural representation of the symmetric groups, A. Young produced representations known as 
the orthogonal form and the seminormal form and gave a straightforward method of calculating the matrices 
representing permutations.  A disadvantage of these representations is that the matrix entries are not 
in general integers, and for many practical purposes, the natural representation is preferable.  Most methods 
for working out the matrices for the natural representation are messy, but this paper gives an approach which 
is simple both to prove and to apply.  Let $T_1, T_2, \dots, T_f$ be the standard tableaux.  For each $\pi \in S_n$, 
form the $f \times f$ matrix $A_\pi$ whose $i,j$ entry is given by the following rule.  If two numbers lie in 
the same row of $\pi T_j$ and in the same column of $T_i$, then the $i,j$ entry in $A_\pi$ is zero.  Otherwise, 
the $i,j$ entry equals the sign of the column permutation for $T_i$ which takes the numbers of $T_i$ to 
the correct rows they occupy in $\pi T_j$.  The matrix representing $\pi$ in the natural representation is then 
$A_I^{-1} A_\pi$, where $I$ is the identity permutation of $S_n$.''

  \begin{figure}[h]
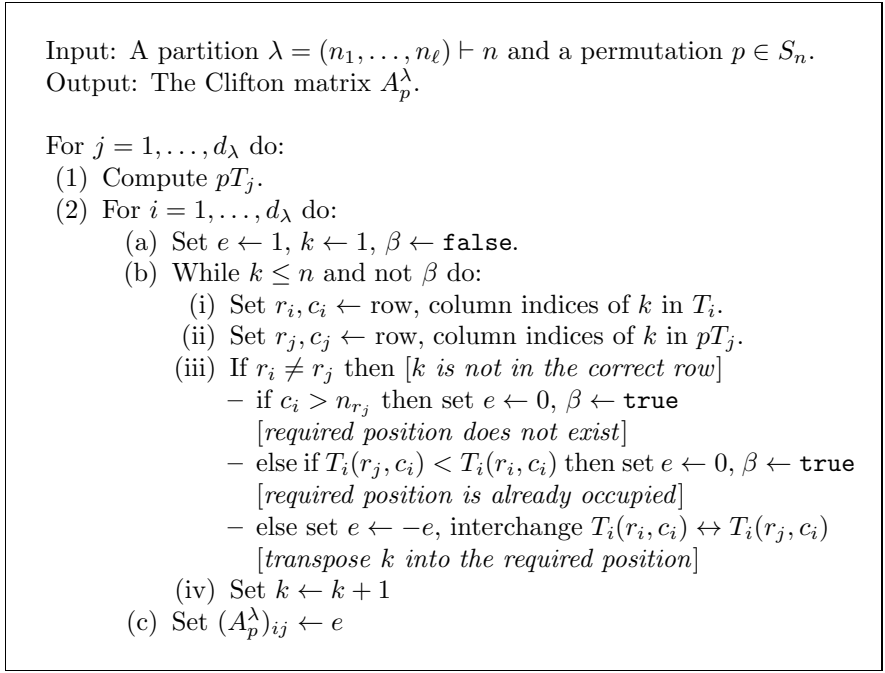

  \begin{mdframed}[style=mystyle]
  \smallskip
  \begin{itemize}
  \item[]
  Input: A partition $\lambda = (n_1,\dots,n_\ell) \vdash n$ and a permutation $p \in S_n$.
  \item[]
  Output: The Clifton matrix $A^\lambda_p$.
  \item[]
  \item[]
  For $j = 1, \dots, d_\lambda$ do:
  \begin{enumerate}
  \item
  Compute $p T_j$.
  \item
  For $i = 1, \dots, d_\lambda$ do:
    \begin{enumerate}
    \item
    Set
    $e \leftarrow 1$,
    $k \leftarrow 1$,
    $\beta \leftarrow \texttt{false}$.
    \item
    While $k \le n$ and not $\beta$ do:
      \begin{enumerate}
      \item
      Set
      $r_i, c_i \leftarrow$
      row, column indices of $k$ in $T_i$.
      \item
      Set
      $r_j, c_j \leftarrow$
      row, column indices of $k$ in $p T_j$.
      \item
      If $r_i \ne r_j$ then
        [$k$ \emph{is not in the correct row}]
        \begin{itemize}
        \item
        if $c_i > n_{r_j}$ then set $e \leftarrow 0$, $\beta \leftarrow \texttt{true}$
        \item[]
        [\emph{required position does not exist}]
        \item
        else if $T_i(r_j,c_i) < T_i(r_i,c_i)$ then set $e \leftarrow 0$, $\beta \leftarrow \texttt{true}$
        \item[]
        [\emph{required position is already occupied}]
        \item
        else set $e \leftarrow -e$, interchange $T_i(r_i,c_i) \leftrightarrow T_i(r_j,c_i)$
        \item[]
        [\emph{transpose} $k$ \emph{into the required position}]
        \end{itemize}
      \item
      Set $k \leftarrow k + 1$
      \end{enumerate}
      \item Set $(A^\lambda_p)_{ij} \leftarrow e$
    \end{enumerate}
  \end{enumerate}
  \end{itemize}
  \medskip
  \end{mdframed}
  \vspace{-3mm}
  \caption{Algorithm to compute the Clifton matrix $A^\lambda_p$}
  \label{cliftonmatrixalgorithm}
  \end{figure}

The Wedderburn decomposition of $\mathbb{F} S_n$ shows that every permutation $p \in S_n$ is
a sum of terms $p^\lambda \in \mathbb{F}S_n$ for $\lambda \vdash n$, and 
each $p^\lambda$ is a linear combination of the $U^\lambda_{ij}$:
  \[
  p
  =
  \sum_\lambda \left( \sum_{i=1}^{d_\lambda} \sum_{j=1}^{d_\lambda} r^\lambda_{ij}(p) U^\lambda_{ij} \right)
  \]

\begin{definition}
We define $R^\lambda(p)$ to be the $d_\lambda \times d_\lambda$ matrix with $(i,j)$ entry $r^\lambda_{ij}(p)$.
We call $R^\lambda(p)$ the \textbf{representation matrix} of $p \in S_n$ for $\lambda \vdash n$.
\end{definition}

\begin{lemma}
We have $U^\lambda_{ii} \, p \, U^\lambda_{jj} = r^\lambda_{ij}(p) U^\lambda_{ij}$.
\end{lemma}

\begin{proposition} \emph{\textbf{(Clifton's theorem)}}
For all $\lambda \vdash n$ and $p \in S_n$ we have 
  \[
  R^\lambda(p) = (A^\lambda_\iota)^{-1} A^\lambda_p.
  \]
\end{proposition}

\begin{proof}
We write $\Epsilon = A^\lambda_\iota$ and denote the entries of $\mathcal{E}^{-1}$ by $\eta_{ij}$.
We have
  \begin{align*}
  U^\lambda_{ii} \, p \, U^\lambda_{jj}
  &=
  \alpha( E_{ii} \mathcal{E}^{-1} ) \, p \, \alpha( E_{jj} \mathcal{E}^{-1} )
  =
  \left( \sum_{k=1}^{d_\lambda} \eta_{ik} E_i s_{ik} \right)
  p
  \left( \sum_{\ell=1}^{d_\lambda} \eta_{j\ell} E_j s_{j\ell} \right)
  \\
  &=
  \sum_{k=1}^{d_\lambda} \sum_{\ell=1}^{d_\lambda}
  \eta_{ik} \, \eta_{j\ell} \, E_i \, s_{ik} \, p \, E_j \, s_{j\ell}
  =
  \sum_{k=1}^{d_\lambda} \sum_{\ell=1}^{d_\lambda}
  \eta_{ik} \, \eta_{j\ell} \, s_{ik} \, E_k \, p \, E_j \, s_{j\ell}
  \\
  &\stackrel{\eqref{eipej}}{=}
  \sum_{k=1}^{d_\lambda} \sum_{\ell=1}^{d_\lambda}
  \eta_{ik} \, \eta_{j\ell} \, s_{ik} \, \esgn^p_{kj} \, s_{kj} \, E_j \, s_{j\ell}
  =
  \sum_{k=1}^{d_\lambda} \sum_{\ell=1}^{d_\lambda}
  \eta_{ik} \, \eta_{j\ell} \, \esgn^p_{kj} \, s_{ik} \, s_{kj} \, E_j \, s_{j\ell}
  \\
  &=
  \sum_{k=1}^{d_\lambda} \sum_{\ell=1}^{d_\lambda}
  \eta_{ik} \, \eta_{j\ell} \, \esgn^p_{kj} \, s_{ik} \, E_k \, s_{kj} \, s_{j\ell}
  =
  \sum_{k=1}^{d_\lambda} \sum_{\ell=1}^{d_\lambda}
  \eta_{ik} \, \eta_{j\ell} \, \esgn^p_{kj} \, E_i \, s_{ik} \, s_{kj} \, s_{j\ell}
  \\
  &=
  \sum_{k=1}^{d_\lambda} \sum_{\ell=1}^{d_\lambda}
  \eta_{ik} \, \eta_{j\ell} \, \esgn^p_{kj} \, E_i \, s_{i\ell}
  =
  \left(
  \sum_{k=1}^{d_\lambda}
  \eta_{ik} \, \esgn^p_{kj}
  \right)
  \left(
  \sum_{\ell=1}^{d_\lambda}
  \eta_{j\ell} \, E_i \, s_{i\ell}
  \right)
  \\
  &=
  \left(
  \sum_{k=1}^{d_\lambda}
  \eta_{ik} \, \esgn^p_{kj}
  \right)
  U_{ij}
  =
  ( A_\iota^{-1} A_p  )_{ij} \, U_{ij}.
  \end{align*}
Therefore $r^\lambda_{ij}(p) = ( A_\iota^{-1} A_p )_{ij}$ for all $i, j$ and so $R^\lambda(p) = A_\iota^{-1} A_p$
as required.
\end{proof}

\begin{example}
For $n = 3$ we have $A^\lambda_\iota=I_{d_\lambda}$ for all $\lambda \vdash 3$, so $R^\lambda_p = A^\lambda_p$ 
for all $p \in S_3$.
Consider $\lambda = 21$ with $d_\lambda = 2$, and $p = 213$.
For $i, j = 1, 2$ we write the tableaux $T_i$ and $p T_j$, and the vertical permutation $q$ (when it exists):
  \[
  \begin{array}{lllll}
  (i,j) = (1,1)
  &\quad
  T_1 =\!\! \begin{array}{c} \young(12,3) \end{array}
  &\quad
  p T_1 =\!\! \begin{array}{c} \young(21,3) \end{array}
  &\quad
  q = \iota
  &\quad
  \epsilon(q) = 1
  \\[9pt]
  (i,j) = (1,2)
  &\quad
  T_1 =\!\! \begin{array}{c} \young(12,3) \end{array}
  &\quad
  p T_2 =\!\! \begin{array}{c} \young(23,1) \end{array}
  &\quad
  q = 321
  &\quad
  \epsilon(q) = -1
  \\[9pt]
  (i,j) = (2,1)
  &\quad
  T_2 =\!\! \begin{array}{c} \young(13,2) \end{array}
  &\quad
  p T_1 =\!\! \begin{array}{c} \young(21,3) \end{array}
  &
  \multicolumn{2}{l}{\quad \text{$q$ does not exist}}
  \\[9pt]
  (i,j) = (2,2)
  &\quad
  T_2 =\!\! \begin{array}{c} \young(13,2) \end{array}
  &\quad
  p T_2 =\!\! \begin{array}{c} \young(23,1) \end{array}
  &\quad
  q = 213
  &\quad
  \epsilon(q) = -1
  \end{array}
  \]
We obtain the following result which agrees with $(\psi^\lambda)^{-1}(213)$ from Example \ref{example.n=3}:
  \[
  A^\lambda_p = \begin{bmatrix}  1 & -1 \\  0 & -1 \end{bmatrix}
  \]
\end{example}

\begin{example}
Consider $n = 5$, the smallest $n$ for which there exists $\lambda \vdash n$ such that
$A^\lambda_\iota \ne I_{d_\lambda}$.
We list the standard tableaux for $\lambda = 32$ in lex order:
  \[
  T_1, \dots, T_5
  =\!\! 
  \begin{array}{c} 
  \young(123,45)
  \quad
  \young(124,35)
  \quad
  \young(125,34)
  \quad
  \young(134,25)
  \quad
  \young(135,24) 
  \end{array}
  \]
Let $p = \iota$ and consider the $(i,j) = (1,5)$ entry of $\Epsilon = A^\lambda_\iota$;
we have $T_i = T_1$ and $p T_j = T_5$.
The required vertical permutation is the transposition $q = 15342$ interchanging 2 and 5,
so $( A^\lambda_\iota )_{15} = -1$.
Similar calculations show that
  \[
  A^\lambda_\iota = I_5 - E_{15},
  \qquad
  ( A^\lambda_\iota )^{-1} = I_5 + E_{15}.
  \]
To illustrate the difference between the Clifton matrix $A^\lambda_p$
and the representation matrix $R^\lambda_p = ( A^\lambda_\iota )^{-1} A^\lambda_p$, consider the 5-cycle
$p = 23451$; in this case we obtain
  \[
  A^\lambda_p
  =
  \left[
  \begin{array}{rrrrr}
  -1 &  0 & 1 & 0 & 0 \\
  -1 &  0 & 0 & 0 & 1 \\
   0 & -1 & 0 & 0 & 0 \\
  -1 &  0 & 0 & 1 & 0 \\
   0 & -1 & 0 & 1 & 0
  \end{array}
  \right]
  \qquad\qquad
  R^\lambda_p
  =
  \left[
  \begin{array}{rrrrr}
  -1 & -1 & 1 & 1 & 0 \\
  -1 &  0 & 0 & 0 & 1 \\
   0 & -1 & 0 & 0 & 0 \\
  -1 &  0 & 0 & 1 & 0 \\
   0 & -1 & 0 & 1 & 0
  \end{array}
  \right].
  \]
\end{example}

%%%%%%%%%%%%%%%%%%%%%%%%%%%%%%%%%%%%%%%%%%%%%%%%%%%%%%%%%%%%%%%%%%%%%%%%%%%%%%%%%%%%%%%%%%%%%%%%%%%%%%%%%%%%
%%%%%%%%%%%%%%%%%%%%%%%%%%%%%%%%%%%%%%%%%%%%%%%%%%%%%%%%%%%%%%%%%%%%%%%%%%%%%%%%%%%%%%%%%%%%%%%%%%%%%%%%%%%%
%%%%%%%%%%%%%%%%%%%%%%%%%%%%%%%%%%%%%%%%%%%%%%%%%%%%%%%%%%%%%%%%%%%%%%%%%%%%%%%%%%%%%%%%%%%%%%%%%%%%%%%%%%%%

\section{Computational methods for studying polynomial identities} \label{allidentities}

Let $A$ be an algebra, not necessarily associative, which is finite dimensional over a field $\mathbb{F}$.
The multiplication in $A$ is a bilinear map $m\colon A \times A \to A$ denoted by $(x,y) \mapsto xy$.
We write $d$ for the dimension of $A$ over $\mathbb{F}$.
If we choose an ordered basis $v_1, \dots, v_d$ of the vector space $A$, then the multiplication in $A$ can
be expressed in terms of the structure constants $c_{ij}^k$ with respect to this basis:
  \[
  v_i v_j = \sum_{k=1}^d c_{ij}^k v_k.
  \]
A polynomial identity satisfied by $A$ is an equation of the form $I \equiv 0$ where $I$ is a nonassociative
noncommutative polynomial (not necessarily multilinear or even homogeneous) which vanishes when arbitrary
elements of $A$ are substituted for the variables in $I$.
We use the symbol $\equiv$ to indicate that the equation holds for all values of the variables.
The polynomial identities satisfied by the algebra $A$ do not depend on the choice of basis.

%%%%%%%%%%%%%%%%%%%%%%%%%%%%%%%%%%%%%%%%%%%%%%%%%%%%%%%%%%%%%%%%%%%%%%%%%%%%%%%%%%%%%%%%%%%%%%%%%%%%%%%%%%%%

\subsection{Historical background}

We denote by $T_X(A)$ the set of polynomial identities in the set of variables $X$ satisfied by the algebra $A$.
The set $T_X(A)$ is an ideal of the free nonassociative algebra $\mathbb{F}\{X\}$ generated by $X$.
Moreover, $T_X(A)$ is a  $T$-ideal: $f(T_X(A)) \subseteq T_X(A)$ for any endomorphism 
$f\colon \mathbb{F}\{X\} \to \mathbb{F}\{X\}$.

\begin{problem}
Specht, 1950 \cite{Specht1950}.
Given a class of algebras, determine whether every algebra in this class has a finite basis,
in the sense that its  $T$-ideal is generated by a finite number of identities.
\end{problem}

Specht originally posed this problem for associative algebras over fields of characteristic zero.
The complete solution was given by Kemer.

\begin{theorem}
\emph{Kemer, 1987 \cite{Kemer1987}}.
Every associative algebra over a field of characteristic zero has a finite basis of identities.
\end{theorem}

Similar results were obtained by Vais and Zelmanov {\cite{VS1989} for 
finitely generated Jordan algebras (1989), and by Iltyakov \cite{Iltyakov1991,Iltyakov1992} for 
finitely generated alternative algebras (1991) and Lie algebras (1992).

If we consider the usual grading of $\mathbb{F}\{X\}$ by total degree, then we can study $T_X(A)_n$ 
for each $n \in \mathbb{N}$,
the homogeneous component of degree $n$ of the $T$-ideal. 
The nonzero elements of $T_X(A)_n$ are the polynomial identities of degree $n$ for $A$.
An important problem is to find the smallest $n$ for which $T(A)_n \ne 0$;
in this case, the nonzero elements of $T_X(A)_n$ are called minimal identities for $A$.
For the simple matrix algebras $M_{n} ({\mathbb F})$, the minimal identities were found
by Amitsur and Levitzki.

\begin{theorem} \label{ALtheorem}
\emph{Amitsur and Levitzki, 1950 \cite{AL1950}}.
The minimal degree of a polynomial identity of $M_{n}({\mathbb F})$ is $2n$.
Every multilinear polynomial identity of degree $2n$ for $M_{n}({\mathbb F})$ is a scalar multiple of
the standard polynomial:
  \[
  s_{2n} (x_1, \dots , x_{2n}) 
  = 
  \sum_{\sigma \in S_{2n}} \epsilon(\sigma) x_{\sigma(1)} \cdots x_{\sigma(2n)}.
  \]
\end{theorem}

Leron \cite{Leron1973} proved (1973) that if char(${\mathbb F}) = 0$ and $n > 2$ then every polynomial identity 
of degree $2n + 1$ for $M_n({\mathbb F})$ is a consequence of $s_{2n}$.
In particular, the identities of degree 7 for $M_3({\mathbb F})$ are consequences of $s_6$.
Drensky and Kasparian \cite{DK1983} found (1983) all identities of degree 8 for $M_3({\mathbb F})$ 
when char(${\mathbb F}) = 0$, and showed that they are consequences of $s_6$; 
see also Bondari \cite{Bondari1993,Bondari1997}.
The  $T$-ideal of identities $T(M_2({\mathbb F}))$ has been studied by many authors;
see Razmyslov \cite{Razmyslov1994} for a survey.
The computational methods used to study polynomial identities of matrices have been described by Benanti and
co-authors \cite{BDDK2003}.

\begin{problem} \label{problem}
Given an ordered basis and structure constants $c_{ij}^k$ for a finite-dimensional algebra $A$ over $\mathbb{F}$,
determine the polynomial identities of degree $\le n$ satisfied by $A$.
In particular, find the minimal identities satisfied by $A$.
\end{problem}

Over a field of characteristic 0, every polynomial identity is equivalent to a set of multilinear identities;
see Zhevlakov and co-authors \cite{ZSSS1982}.
Hence in characteristic 0, we may restrict our study to multilinear identities:
equations of the form $I( x_1, \dots, x_n ) \equiv 0$ where $I( x_1, \dots, x_n )$ is a linear combination of monomials
in which each of the $n$ variables $x_1, \dots, x_n$ occurs exactly once.
So each term of $I( x_1, \dots, x_n )$ consists of a coefficient from $\mathbb{F}$ and a nonassociative monomial,
which is a permutation of the variables $x_1, x_2, \dots, x_n$ together with an association type
(placement of parentheses) which indicates the order in which the multiplications are performed.
If there are $t=t(n)$ association types in degree $n$, then $I( x_1, \dots, x_n )$ can be written as a sum of $t$ summands
$I_1 + I_2 + \cdots + I_t$, where the elements in each summand have the same association type.
Within each summand, the monomials differ only in the permutation of the variables,
and so each summand can be regarded as an element of the group algebra $\mathbb{F} S_n$.
We can therefore regard $I( x_1, \dots, x_n )$ as an element of $\mathbb{F} S_n \oplus \cdots \oplus \mathbb{F} S_n$,
the direct sum of $t$ copies of $\mathbb{F} S_n$.

This approach to polynomial identities was introduced independently in 1950 
by Malcev \cite{Malcev1950} and Specht \cite{Specht1950}.
In the 1970's, Regev developed this theory further, with particular application to associative PI algebras; 
see for instance \cite{Regev1978,Regev1988}.
Around the same time, the computational implementation of this theory was initiated by Hentzel \cite{Hentzel1977a,Hentzel1977b}.
(Two of the present authors learned about the application of this theory to polynomial identities through working with Hentzel.)

%%%%%%%%%%%%%%%%%%%%%%%%%%%%%%%%%%%%%%%%%%%%%%%%%%%%%%%%%%%%%%%%%%%%%%%%%%%%%%%%%%%%%%%%%%%%%%%%%%%%%%%%%%%%

\subsection{Multilinear polynomial identities satisfied by an algebra}

Consider a countable set $X =\{x_1, x_2, \dots \}$ of variables.

\begin{definition}
We construct the set $M(X)$ of \textbf{nonassociative monomials} in the variables $X$ inductively as follows
\cite{ZSSS1982}:
\begin{itemize}
\item
$X \subset M(X)$;
\item
if $x, y \in X$ and $v, w \in M(X) \setminus X$ then $xy, \, x(v), \, (v)x, \, (v)(w) \in M(X)$.
\end{itemize}
An \textbf{association type} is a placement of parentheses on a monomial of $M(X)$.
\end{definition}

There is a natural total order on association types of degree $n$ defined inductively, using 
the fact that every nonassociative monomial has a unique factorization $x = yz$.
We fix a symbol $\ast$ and think of an association type as a monomial in which every variable equals $\ast$.
If $x = yz$ and $x' = y'z'$ are monomials of degree $n$, then $x \prec x'$ if and only if 
either $y \prec y'$ or $y = y'$ and $z \prec z'$.

\begin{example}
For $n = 3$ we have two types: $(\ast\ast)\ast$ and $\ast(\ast\ast)$.
For $n = 4$ we have five types:
$((\ast\ast)\ast)\ast$, $(\ast(\ast\ast))\ast$, $(\ast\ast)(\ast\ast)$, $\ast((\ast\ast)\ast)$, $\ast(\ast(\ast\ast))$.
\end{example}

\begin{lemma}
The number of distinct association types of degree $n$ in $\mathbb{F}\{X\}$
is equal to the Catalan number (with the index shifted by 1):
  \[
  C_{n-1} = \frac{1}{n} \binom{2n{-}2}{n{-}1}.
  \]
\end{lemma}

The numbers $C_n$ grow very rapidly; here we present the first 12:
  \[
  \begin{array}{l|rrrrrrrrrrrr}
  n & 1 & 2 & 3 & 4 & 5 & 6 & 7 & 8 & 9 & 10 & 11 & 12
  \\
  C_n & 1 & 1 & 2 & 5 & 14 & 42 & 132 & 429 & 1430 & 4862 & 16796 & 58786
  \end{array}
  \]

\begin{example}
If $A$ is an associative algebra, then the placement of parentheses does not affect the product,
and so we only need to choose one association type in each degree as the normal form.
If necessary, we choose the right-normed product $x_1 ( x_2 ( \cdots ( x_{n-1} x_n ) \cdots ) )$,
using the identity permutation of the variables.
But usually we can omit the parentheses and write simply $x_1 x_2 \cdots x_{n-1} x_n$.
In an associative algebra, any two multilinear monomials of degree $n$ in $n$ variables differ
only by the permutation of the variables, and so a multilinear polynomial identity in degree $n$ 
can be regarded as an element of the group algebra $\mathbb{F} S_n$.
\end{example}

\begin{example}
If $A$ is commutative (such as a Jordan algebra) or anticommutative (such a Lie algebra), then 
the association types are not independent.
For example, we have $(ab)c = \pm c(ab)$.
In these cases, the Wedderburn-Etherington numbers (oeis.org/A001190) enumerate the association types:
1, 1, 1, 2, 3, 6, 11, 23, 46, \dots
\end{example}

\begin{definition}
If $\mathbb{F}$ is a field, then we write $\mathbb{F}\{X\}$ for the vector space  over $\mathbb{F}$ with basis $M(X)$.
We define a multiplication on $\mathbb{F}\{X\}$ by extending bilinearly the product in $M(X)$, and call this the 
\textbf{free nonassociative algebra} generated by $X$ over $\mathbb{F}$.
Its elements are called \textbf{nonassociative polynomials} in the variables $X$.
\end{definition}

\begin{definition}
Let $A$ be an algebra over $\mathbb{F}$ (not necessarily associative). 
A nonassociative polynomial $f = f(x_1, \dots, x_n) \in \mathbb{F}\{X\}$ is
called a \textbf{polynomial identity} of $A$ if $f(a_1, \dots, a_n) = 0$ for all $a_1, \dots, a_n \in A$.
We often write this more compactly as $f \equiv 0$.
If every variable $x_1, \dots, x_n$ appears exactly once in every monomial of $f$ then
$f$ is called \textbf{multilinear} polynomial identity.
\end{definition}

We now explain the basic \textbf{fill-and-reduce} algorithm to find the multilinear polynomial identities 
of degree $n$ for an algebra $A$ 
of dimension $d < \infty$ over a field $\mathbb{F}$.
We choose a basis for $A$ and express elements of $A$ as vectors in $\mathbb{F}^d$.
In degree $n$ there are $t=t(n)$ association types and $n!$ permutations of the variables,
for a total of $tn!$ distinct monomials; we fix once and for all a total order on these monomials.
A polynomial identity $I( x_1, \dots, x_n )$ is a linear combination of these $tn!$ monomials,
with coefficients in $\mathbb{F}$.
Let $E(n)$ be a matrix with $tn!$ columns and $tn! + d$ rows, consisting of
a $tn! \times tn!$ upper block and a $d \times tn!$ lower block.
We generate $n$ pseudorandom elements $a_1, \dots, a_n \in A$.
We evaluate the $tn!$ monomials by setting $x_i = a_i$ ($i = 1, \dots, n$) and
obtain a sequence $r_j$ ($j = 1, \dots, tn!$) of elements of $A$.
For each $j$ we put the coefficient vector of $r_j$ into the $j$th column of the lower block.
The $d$ rows of the lower block consist of linear constraints on the coefficients of
the general multilinear polynomial identity $I( x_1, \dots, x_n )$.
We compute the row canonical form $\mathrm{RCF}(E(n))$, so the lower block becomes zero.
We repeat this process of generating pseudorandom elements of $A$, filling the lower block,
and reducing the matrix until the rank of $E(n)$ stabilizes.
At this point, we write $a$ for the nullity; the nullspace consists of the coefficient vectors of a canonical set 
of generators for the multilinear polynomial identities satisfying the constraints imposed at each step, that is,
the multilinear polynomial identities in degree $n$ satisfied by $A$.
We compute the canonical basis of the nullspace by setting the free variables equal to the standard basis vectors
and solving for the leading variables.
We then put these canonical basis vectors into another matrix of size $a \times tn!$, and compute its RCF,
which we denote by $[ \mathrm{All}(n) ]$.
We call the row space of this matrix $\mathrm{All}(n)$; this is the vector space of all multilinear identities
of degree $n$ satisfied by $A$.
This method is only practical when the number $tn!$ of monomials is relatively small.

\begin{example} \label{m22degree4}
We find the polynomial identities of degree 4 for $A = M_2(\mathbb{F})$,
the 4-dimensional associative algebra of $2 \times 2$ matrices over $\mathbb{F}$.
We construct a $28 \times 24$ zero matrix $E(4)$ and repeat the following process:
  \begin{itemize}
  \item
  generate pseudorandom $2 \times 2$ matrices $a_1, a_2, a_3, a_4$ over $\mathbb{F}$;
  \item
  evaluate $m^{(j)} = a_{p_j(1)} a_{p_j(2)} a_{p_j(3)} a_{p_j(4)}$ for all 
  $p_j \in S_4 = \{ p_1, \dots, p_{24} \}$;
  \item
  for $1 \le j \le 24$, store $m^{(j)}$ in the last 4 positions of column $j$ of $E(4)$:
  \[
  E(4)_{25,j} \leftarrow m^{(k)}_{11}, \;
  E(4)_{26,j} \leftarrow m^{(k)}_{12}, \;
  E(4)_{27,j} \leftarrow m^{(k)}_{21}, \; 
  E(4)_{28,j} \leftarrow m^{(k)}_{22};
  \]
  \item
  compute the row canonical form $\mathrm{RCF}(E(4))$.
  \end{itemize}
The first 6 iterations produce ranks 4, 8, 12, 16, 20, 23 and the rank remains 23 for the 
next 10 iterations.  
Hence the nullity is 1, and a basis for the nullspace consists of the coefficient vector
of the standard identity of degree $4$:
  \begin{equation}
  \label{standardidentity4}
  s_4(x_1,x_2,x_3,x_4) = \sum_{p \in S_4} \epsilon(p) \, x_{p(1)} x_{p(2)} x_{p(3)} x_{p(4)} \equiv 0.
  \end{equation}
This is the Amitsur-Levitzki identity from Theorem \ref{ALtheorem} in the case $n=2$.
\end{example}

%%%%%%%%%%%%%%%%%%%%%%%%%%%%%%%%%%%%%%%%%%%%%%%%%%%%%%%%%%%%%%%%%%%%%%%%%%%%%%%%%%%%%%%%%%%%%%%%%%%%%%%%%%%%

\subsection{Consequences of polynomial identities in higher degrees}

When computing the multilinear polynomial identities satisfied by an algebra $A$,
we often find that many of the identities in degree $n$ are consequences of known identities of lower degrees,
so they do not provide any new information.
We want the identities in degree $n$ which cannot be expressed in terms of known identities of lower degrees.

\begin{definition} \label{liftingdefinition}
Let $I( x_1, \dots, x_n )$ be a multilinear nonassociative polynomial of degree $n$.
There are $n{+}2$ \textbf{consequences} of this polynomial in degree $n{+}1$, namely
$n$ substitutions obtained by replacing $x_i$ by $x_i x_{n+1}$ ($i = 1, \dots, n$)
and two multiplications of $I$ by $x_{n+1}$ (on the right and the left):
  \begin{align*}
  &
  I( x_1 x_{n+1}, \dots, x_n ),
  \quad \dots \quad
  I( x_1, \dots, x_i x_{n+1}, \dots, x_n ),
  \quad \dots \quad
  I( x_1, \dots, x_n x_{n+1} ),
  \\
  &
  I( x_1, \dots, x_i, \dots, x_n ) x_{n+1},
  \qquad
  x_{n+1} I( x_1, \dots, x_i, \dots, x_n ).
  \end{align*}
If $I \equiv 0$ is a polynomial identity for an algebra $A$, then so are its consequences.  
\end{definition}

\begin{lemma}
Every multilinear polynomial of degree $n{+}1$ in the $T$-ideal generated by $I(x_1,\dots,x_n) \in \mathbb{F}\{X\}$
is a linear combination of permutations of the $n{+}2$ consequences in Definition \ref{liftingdefinition}.
\end{lemma}

\begin{proof}
By definition, the $T$-ideal generated by $I$ in $\mathbb{F}\{X\}$ is the ideal containing
$I$ which is invariant under all endomorphisms of $\mathbb{F}\{X\}$.
The $n$ substitutions correspond to invariance under endomorphisms, and the two multiplications
correspond to invariance under right and left multiplication.
\end{proof}

\begin{example} \label{alternative4}
The algebra $\mathbb{O}$ of octonions which we will study in detail later is an example of an alternative algebra.
Alternative algebras are defined by the left and right alternative identities
$(x,x,y) \equiv 0$ and $(x,y,y) \equiv 0$, where $(x,y,z) = (xy)z - x(yz)$ is the associator.
Over a field of characteristic $\ne 2$, these two identities are equivalent to their linearized forms:
  \[
  (x,z,y) + (z,x,y) \equiv 0,
  \qquad
  (x,y,z) + (x,z,y) \equiv 0.
  \]
Each of these identities has five consequences in degree 4; for example, from the left alternative identity we obtain
  \begin{align*}
  &
  (xw,z,y) + (z,xw,y) \equiv 0, \;\;
  (x,z,yw) + (z,x,yw) \equiv 0, \;\;
  (x,zw,y) + (zw,x,y) \equiv 0,
  \\
  &
  (x,z,y)w + (z,x,y)w \equiv 0, \;\;
  w(x,z,y) + w(z,x,y) \equiv 0.
  \end{align*}
\end{example}

The vector space $\mathrm{All}(n)$ of all multilinear polynomial identities of degree $n$ satisfied by an algebra $A$
is a subspace of the multilinear space of degree $n$ in the free nonassociative algebra $\mathbb{F}\{X\}$ where
$X = \{ x_1, \dots, x_n \}$.
Since $\mathrm{All}(n)$ is invariant under permutations of the variables, we can regard $\mathrm{All}(n)$ 
as a left $S_n$-module with action given by permuting the subscripts of the variables:
$\sigma \cdot f(x_1, \dots, x_n) = f(x_{\sigma(1)}, \dots, x_{\sigma(n)})$.
We can also consider $\mathrm{All}(n)$ as a submodule of $\mathrm{Bin}(n)$, the degree $n$ component of the operad 
$\mathrm{Bin}$ generated by one nonassociative binary operation with no symmetry. 
The operad $\mathrm{Bin}$ is non-symmetric, but the operad governing the algebra $A$ has the quotient 
$\mathrm{Bin}(n) / \mathrm{All}(n)$ as its $S_n$-module in degree $n$, and hence may be symmetric or non-symmetric,
depending on the properties of the identities satisfied by $A$.

For a given algebra $A$, the consequences in degree $n$ of the identities of degrees $< n$ generate a submodule 
$\mathrm{Old}(n) \subseteq \mathrm{All}(n)$.
We now explain the basic \textbf{module generators} algorithm to find a canonical set of $S_n$-module generators 
for $\mathrm{Old}(n)$.
We assume by induction that we have already determined a set of $S_{n-1}$-module generators for $\mathrm{All}(n{-}1)$.
The consequences of these generators in degree $n$ form a set $O(n)$ of $S_n$-module generators for $\mathrm{Old}(n)$.
We construct a $(tn!+n!) \times tn!$ matrix $C(n)$ consisting of a $tn! \times tn!$ upper block and a $n! \times tn!$ 
lower block (as before, $t=t(n)$ is the number of association types in degree $n$).
Using the lex order on permutations, we write $\sigma_i$ for the $i$th element of $S_n$.
We take an identity $I \in O(n)$ and for $i = 1, \dots,n!$ we put the coefficient vector of $\sigma \cdot I$ 
into the $i$th row of the lower block.
The $n!$ rows of the lower block then contain all the permutations of $I$, and hence they span the $S_n$-module
generated by $I$.
We compute $\mathrm{RCF}(C(n))$ so the lower block becomes zero.
We repeat this process for each $I \in O(n)$.
At the end, the nonzero rows of $\mathrm{RCF}(C(n))$ form a matrix $[ \mathrm{Old}(n) ]$ which contains the coefficient 
vectors of a canonical set of $S_n$-module generators for $\mathrm{Old}(n)$.

We compare the $S_n$-modules $\mathrm{Old}(n)$ and $\mathrm{All}(n)$ to determine whether there exist
new multilinear identities in degree $n$ satisfied by $A$; that is, identities which do not follow from those of 
degrees $< n$.
To do this, we compare the reduced matrices $[ \mathrm{Old}(n) ]$ and $[ \mathrm{All}(n) ]$;
we denote their ranks by $r_\text{old}$ and $r_\text{all}$.
If $r_\text{old} = r_\text{all}$ then we must have $[ \mathrm{Old}(n) ] = [ \mathrm{All}(n) ]$:
every identity in degree $n$ satisfied by $A$ follows from identities of lower degrees.
If $r_\text{old} \ne r_\text{all}$ then since $\mathrm{Old}(n) \subseteq \mathrm{All}(n)$ we must have
$r_\text{old} < r_\text{all}$, and the row space of $[ \mathrm{Old}(n) ]$ must be a subspace of the row space of 
$[ \mathrm{All}(n) ]$.
The difference $r_\text{all} - r_\text{old}$ is the dimension of the $S_n$-module of new identities in degree $n$.

\begin{definition}
The \textbf{new identities} satisfied by $A$ in degree $n$ are the nonzero elements of the quotient module 
$\mathrm{New}(n) = \mathrm{All}(n) / \mathrm{Old}(n)$.
\end{definition}

\begin{definition} \label{leading1s}
If $X$ is a matrix in RCF, we write $\texttt{leading}(X)$ for the set of ordered pairs
$(i,j)$ such that $X$ has a leading 1 in row $i$ and column $j$.
We write $\texttt{jleading}(X) = \{ \, j \mid (i,j) \in \texttt{leading}(X) \}$.
\end{definition}

We find $S_n$-module generators for $\mathrm{New}(n)$, by calculating the set difference 
  \[
  \texttt{jleading}([\mathrm{All}(n)]) \setminus \texttt{jleading}([\mathrm{Old}(n)])
  =
  \big\{ j_1, \dots, j_r \big\} \quad (r = r_\text{all} - r_\text{old}).
  \]
For $s = 1, \dots, r$ we define $i_s$ by $( i_s, j_s ) \in \texttt{leading}([\mathrm{All}(n)])$.

\begin{lemma} \label{newgenerators}
Rows $i_1, \dots, i_r$ of $[ \mathrm{All}(n) ]$ are the coefficient vectors of 
the canonical generators of $\mathrm{New}(n)$.
\end{lemma}

\begin{example} \label{m22degree5}
To illustrate these concepts, we extend the results of Example \ref{m22degree4} regarding $2 \times 2$ matrices
from degree 4 to degree 5.

To find all the identities we proceed as before, with some obvious changes: 
the matrix $E(5)$ has size $124 \times 120$; each iteration generates 5 pseudorandom matrices;
there are 120 permutations to evaluate.
We find that the rank increases by 4 for each of the first 22 iterations, but the next iteration produces rank 91,
and this remains constant for the next 10 iterations.
Thus the nullspace of $E(5)$ has dimension 29; this is the $S_5$-module $\mathrm{All}(5)$, 
consisting of the coefficient vectors of all identities in degree 5 satisfied by $2 \times 2$ matrices.

To find which of these identities are new, we need to generate all the consequences in degree 5 of the standard 
identity \eqref{standardidentity4}.
Every consequence is a linear combination of permutations of these 6 generators:
  \[
  \begin{array}{lll}
  s_4(x_1x_5,x_2,x_3,x_4), &\qquad
  s_4(x_1,x_2x_5,x_3,x_4), &\qquad
  s_4(x_1,x_2,x_3x_5,x_4),
  \\[2pt]
  s_4(x_1,x_2,x_3,x_4x_5), &\qquad
  x_5 s_4(x_1,x_2,x_3,x_4), &\qquad
  s_4(x_1,x_2,x_3,x_4) x_5.
  \end{array}
  \]  
We construct a $240 \times 120$ zero matrix $C(5)$ and do the following for each generator:
  \begin{itemize}
  \item
  Set $i \leftarrow 120$.
  \item
  For each permutation $p \in S_5$ do:
    \begin{itemize}
    \item
    Set $i \leftarrow i + 1$.
    \item
    For each term $c m$ in the generator, where $c = \pm 1$, $m = x_{q(1)} \cdots x_{q(5)}$,
    let $j$ be the index of $pq$ in the lex-ordering on $S_5$, and set $C(5)_{ij} \leftarrow c$.
    \end{itemize}
  \item
  Compute the row canonical form $\mathrm{RCF}(C(5))$.
  \end{itemize}
After all 6 generators have been processed, the rank of $C(5)$ is 24; its row space is the $S_5$-module
$\mathrm{Old}(5)$.
Combining this result with that of the previous paragraph, we see that the quotient module $\mathrm{New}(5)$
has dimension 5.

It remains to find generators for $\mathrm{New}(5)$.
From $\mathrm{RCF}(E(5))$ we extract a basis for its nullspace, and 
sort these 29 vectors by increasing Euclidean norm (from 18 to 74).
Starting with $\mathrm{RCF}(C(5))$ we apply the same module generators algorithm to these 29 vectors,
and find that the first vector increases the rank from 24 to 29.
Hence (the coset of) this single vector is a generator for $\mathrm{New}(5)$;
this vector has 18 (nonzero) terms, and all coefficients are $\pm 1$.

We can obtain slightly better results using the LLL algorithm for lattice basis reduction \cite{BP2009b}.
This depends on the fact that the nonzero entries of $\mathrm{RCF}(E(5))$ are all integers ($\pm 1, \pm 2$).
We compute a $120 \times 120$ integer matrix $U$ with determinant $\pm 1$ such that $U E(5)^t$ is 
the Hermite normal form of the transpose of $E(5)$. 
Then the bottom 29 rows of $U$ form a lattice basis for the integer nullspace of $E(5)$.
We sort these vectors by increasing Euclidean norm (from 16 to 34), and proceed as in the previous paragraph.
The first vector increases the rank from 24 to 29, and is the coefficient vector of the linearization of 
the Hall identity:
  \[
  [ \, [x_1,x_2] \circ [x_3,x_4], \, x_5 \, ] \equiv 0,
  \]
where $[x,y] = xy - yx$ is the Lie bracket and $x \circ y = xy + yx$ is the Jordan product.
Drensky \cite{Drensky1981} has shown that $s_4 \equiv 0$ and the Hall identity $[[x,y]^2,z] \equiv 0$ generate 
the $T$-ideal of identities satisfied by $2 \times 2$ matrices over fields of characteristic 0.
\end{example}

%%%%%%%%%%%%%%%%%%%%%%%%%%%%%%%%%%%%%%%%%%%%%%%%%%%%%%%%%%%%%%%%%%%%%%%%%%%%%%%%%%%%%%%%%%%%%%%%%%%%%%%%%%%%

\subsection{Representations of $S_n$ and multilinear identities in degree $n$}

We can use the representation theory of the symmetric group to break down these computations into smaller pieces,
one for each irreducible representation of $S_n$.
This significantly reduces the sizes of the matrices involved.

Fix $\lambda \vdash n$ with irreducible representation of dimension $d_\lambda$.
Let $E^\lambda_{ij}$ for $i, j = 1, \dots, d_\lambda$ be the $d_\lambda \times d_\lambda$ matrix units.
It suffices to consider only the matrix units in the first row, in the following sense.

\begin{lemma} \label{firstrow}
Let $M_\lambda$ be an irreducible submodule of type $\lambda$ in the left regular representation $\mathbb{F} S_n$.
Then there exists a generator $f \in M_\lambda$ such that its matrix form $\phi_\lambda(f)$ is in RCF and has rank 1
(the only nonzero row is the first).
\end{lemma}

\begin{proof}
In the left regular representation, row $i$ can be moved to row 1 by left-multiplying by the element of 
$\mathbb{F} S_n$ which is the image under $\psi$ of the elementary matrix which transposes row 1 and row $i$.
Recall that the matrix units in row $i$ are linear combinations of the elements $E_i s_{ij}$:
combine Definitions \ref{defbigeps}, \ref{defU}, \ref{defpsi} and equation \eqref{defalpha}.
We can left-multiply by any $p \in S_n$ and obtain another element in the same matrix algebra. 
In particular, if $p = s_{1i}$ then using Proposition \ref{propconj} we obtain 
$s_{1i} E_i s_{ij} = E_1 s_{1i} s_{ij} = E_1 s_{1j}$.
Thus left-multiplication by $s_{1i}$ moves the matrix units in row $i$ to row 1.
The other rows are zero by the irreducibility assumption.
\end{proof}

Let $t = t(n)$ be the number of association types in degree $n$.
In the direct sum of $t$ copies of the left regular representation,
the $\lambda$-component is isomorphic to the direct sum of $t$ copies of the full matrix algebra $M_{d_\lambda}(\mathbb{F})$.
We construct a matrix $M$ of size $( t d_\lambda + d ) \times t d_\lambda$, consisting of an upper block
of size $t d_\lambda \times t d_\lambda$ and a lower block of size $d \times t d_\lambda$.
The multilinear associative polynomial $U^\lambda_{1j}$ of degree $n$ is the image under $\psi$ 
of the matrix unit $E^\lambda_{1j}$. 

\begin{definition}
For $k = 1, \dots, t$ we write $[ U^\lambda_{1j} ]_k$ for the multilinear nonassociative
polynomial obtained by applying association type $k$ to every term of $U^\lambda_{1j}$.
\end{definition}

Given $n$ pseudorandom elements of the algebra $A$, we can evaluate $[ U^\lambda_{1j} ]_k$ 
using the structure constants of $A$ to obtain another element of $A$.
We do this for each $k = 1, \dots, t$ and each $j = 1, \dots, d_\lambda$ to obtain a sequence of $t d_\lambda$ elements of $A$,
which we regard as column vectors of dimension $d$.
We store each of these column vectors in the corresponding column of the lower block of $M$,
and then compute the $\mathrm{RCF}(M)$.
We repeat this fill-and-reduce process until the rank of $M$ stabilizes;
at this point, the nullspace of $M$ contains the coefficient vectors of 
the polynomial identities satisfied by $A$ in the component of
$( \mathbb{F} S_n )^t$ corresponding to partition $\lambda$.
We compute the canonical basis of the nullspace, and call its dimension $a_\lambda$.
We put the basis vectors into another matrix of size $a_\lambda \times t d_\lambda$, and compute its RCF.
This matrix, denoted $\texttt{allmat}(\lambda)$, 
contains the canonical form of the polynomial identities for $A$ in partition $\lambda$.

We need to compare $\texttt{allmat}(\lambda)$ with the representation matrix for the consequences 
of known identities from lower degrees.
We construct a matrix of size $\ell d_\lambda \times t d_\lambda$ consisting of $d_\lambda \times d_\lambda$ blocks
where $\ell$ is the number of consequences.
The block in position $(i,j)$ where $i = 1, \dots, \ell$ and $j = 1, \dots, t$ is the representation matrix
for the terms of $i$th consequence in association type $j$.
We compute the RCF of this matrix, and call its rank $o_\lambda$.
We denote the resulting $o_\lambda \times t d_\lambda$ matrix of full rank by $\texttt{oldmat}(\lambda)$;
this contains the canonical form of all the consequences in partition $\lambda$.

Since the row space of $\texttt{oldmat}(\lambda)$ is a subspace of the row space of $\texttt{allmat}(\lambda)$
we have $o_\lambda \le a_\lambda$.
Furthermore, $\texttt{oldmat}(\lambda) = \texttt{allmat}(\lambda)$ if and only if $o_\lambda = a_\lambda$; 
in this case, there are no new identities for the algebra $A$ in partition $\lambda$.
We have
  \[
  \texttt{jleading}( \, \texttt{oldmat}(\lambda) \, )
  \, \subseteq \,
  \texttt{jleading}( \, \texttt{allmat}(\lambda) \, ).
  \]
The rows of $\texttt{allmat}(\lambda)$ whose leading 1s occur in the columns with indices in
  \[
  \texttt{jleading}( \, \texttt{allmat}(\lambda) \, )
  \setminus
  \texttt{jleading}( \, \texttt{oldmat}(\lambda) \, ),
  \]
represent new identities for the algebra $A$ in partition $\lambda$.
(This is the representation theoretic version of Lemma \ref{newgenerators}.)

Consider one of the rows representing a new identity:
  \[
  [ \,
  c^\lambda_{11}, \dots, c^\lambda_{1d_\lambda},
  \quad \dots, \quad
  c^\lambda_{k1}, \dots, c^\lambda_{kd_\lambda},
  \quad \dots, \quad
  c^\lambda_{t1}, \dots, c^\lambda_{td_\lambda}
  \, ]
  \qquad
  (1 \le k \le t).
  \]
As explained, we may assume that this is row 1 of the matrix, and so we can regard it as representing
a linear combination of the elements $[ U^\lambda_{1j} ]_k$ where $1 \le k \le t$ and $1 \le j \le d_\lambda$, 
which gives an explicit form of the new identity:
  \[
  \sum_{k=1}^t
  \sum_{j=1}^{d_\lambda}
  c^\lambda_{k,j} \,
  [ U^\lambda_{1j} ]_k
  \equiv 0.
  \]
In general, identities of this form have a very large number of terms, when fully expanded as elements of
$\mathbb{F}S_n$, especially when $n$ becomes large.

%%%%%%%%%%%%%%%%%%%%%%%%%%%%%%%%%%%%%%%%%%%%%%%%%%%%%%%%%%%%%%%%%%%%%%%%%%%%%%%%%%%%%%%%%%%%%%%%%%%%%%%%%%%%

\subsection{The membership problem for  $T$-ideals}

A basic question about polynomial identities satisfied by an algebra is the following.

\begin{problem} \label{basicproblem}
Let $f^1, \dots, f^k$ and $f$ be multilinear polynomial identities of degree $n$ satisfied by an algebra $A$. 
Does $f$ belong to the $S_n$-module generated by $f^1, \dots , f^k$?
Equivalently, is $f$ a linear combination of permutations of $f^1, \dots , f^k$?
\end{problem}

Let $\phi_\lambda \colon \mathbb{F} \, S_n \to  M_{d_\lambda}( \mathbb{F} )$ be the projection onto
the $\lambda$-component in the Wedderburn decomposition \eqref{iso}.
Let $f = f_1 + \cdots + f_t$ be the decomposition of $f \in (\mathbb{F}S_n)^t$ into terms corresponding 
to the $t = t(n)$ association types.

\begin{definition} \label{philambda}
The \textbf{representation matrix} of $f$ for $\lambda$ equals:
  \[
  \phi_\lambda( f )
  =
  \big[ \quad
  \phi_\lambda(f_1)
  \quad \vert \quad
  \phi_\lambda(f_2)
  \quad \vert \quad
  \cdots
  \quad \vert \quad
  \phi_\lambda(f_{t-1})
  \quad \vert \quad
  \phi_\lambda(f_t)
  \quad \big ]
  \]
More generally, the representation matrix for a sequence of identities $f^1, \dots , f^k$ is obtained
by stacking the matrices $\phi_\lambda(f^1), \dots, \phi_\lambda(f^k)$:
  \[
  \phi_\lambda( f^1, \dots , f^k )
  =
  \left[
  \begin{array}{c}
  \phi_\lambda(f^1) \\ \phi_\lambda(f^2) \\ \vdots \\ \phi_\lambda(f^k)
  \end{array}
  \right]
  =
  \left[
  \begin{array}{ccccc}
  \phi_\lambda(f^1_1) & \phi_\lambda(f^1_2) & \cdots & \phi_\lambda(f^1_{t-1}) & \phi_\lambda(f^1_{t})  
  \\[6pt]
  \phi_\lambda(f^2_1)   & \phi_\lambda(f^2_2)  & \cdots  & \phi_\lambda(f^2_{t-1})  & \phi_\lambda(f^2_{t}) 
  \\
  \vdots & \vdots & \ddots & \vdots & \vdots 
  \\
  \phi_\lambda(f^k_1)  & \phi_\lambda(f^k_2)  & \cdots  & \phi_\lambda(f^k_{t-1})  & \phi_\lambda(f^k_{t})  
  \end{array}
  \right]
  \]
\end{definition}

\begin{proposition} \label{Solutionofbasicproblem}
Let $f^1, \dots, f^k$ and $f$ be multilinear polynomial identities of degree $n$.
Then the following conditions are equivalent:
  \begin{itemize}
  \item
  $f$ belongs to the $S_n$-module generated by $f^1, \dots , f^k$
  \item
  the matrices $\phi_\lambda( f^1, \dots , f^k )$ and $\phi_\lambda( f^1, \dots , f^k, f )$ have the same row space
  \item
  the matrices $\phi_\lambda( f^1, \dots , f^k )$ and $\phi_\lambda( f^1, \dots , f^k, f )$ have the same RCF
  \item
  the matrices $\phi_\lambda( f^1, \dots , f^k )$ and $\phi_\lambda( f^1, \dots , f^k, f )$ have the same rank  
  \end{itemize}
\end{proposition}

\begin{example} \label{altdeg4}
Every alternative algebra $A$ satisfies the multilinear identity
  \[
  f(x,y,z,t) = (xy,z,t)  + (x,y,[z,t]) - x(y,z,t) - (x,z,t)y \equiv 0.
  \]
To prove this we need to verify that $f$ is a consequence of the alternative laws.
Assuming $\mathrm{char}(\mathbb{F}) \ne 2$, the alternative laws are equivalent to their linearizations:
  \[
  (x,y,z) + (y,x,z) \equiv 0, \qquad\qquad (x,y,z) + (x,z,y) \equiv 0.
  \]
The consequences of these identities in degree 4 are as follows;
some follow from others using the alternative laws:
  \begin{alignat*}{2}
  f^1 = (xt,y,z) + (y,xt,z) &\equiv 0, &\qquad\qquad f^6 = (xt,y,z) + (xt,z,y) &\equiv 0, \\
  f^2 = (x,yt,z) + (yt,x,z) &\equiv 0, &\qquad\qquad f^7 = (x,yt,z) + (x,z,yt) &\equiv 0, \\
  f^3 = (x,y,zt) + (y,x,zt) &\equiv 0, &\qquad\qquad f^8 = (x,y,zt) + (x,zt,y) &\equiv 0, \\
  f^4 = (x,y,z)t + (y,x,z)t &\equiv 0, &\qquad\qquad f^9 = (x,y,z)t + (x,z,y)t &\equiv 0, \\
  f^5 = t(x,y,z) + t(y,x,z) &\equiv 0, &\qquad\qquad f^{10} = t(x,y,z) + t(x,z,y) &\equiv 0.
  \end{alignat*}
In degree 4, there are $t = 5$ association types. 
For each $\lambda \vdash 4$ we use Clifton's algorithm to calculate the matrices 
  \[
  M_\lambda = \phi_\lambda( f^1, \dots , f^{10} ),
  \qquad
  N_\lambda = \phi_\lambda( f^1, \dots , f^{10}, f ),
  \]
and compute their RCFs.
For example, when $\lambda = 22$ we have $d_\lambda = 2$ and so the matrix $M_\lambda$ has size
$20 \times 10$ and $N_\lambda$ has size $22 \times 10$.
We display $N_\lambda$ and its RCF, which coincides with the RCF of $M_\lambda$:

  \smallskip
  
  \begin{center}
  \begin{minipage}{\linewidth}
  \small
  \[
  \left[
  \begin{array}{rr|rr|rr|rr|rr}
  -1 &\nnn  1 &\nn  1 &\nnn  0 &\nn  1 &\nnn -1 &\nn -1 &\nnn  0 &\nn  0 &\nnn  0 \\
  -1 &\nnn  0 &\nn  0 &\nnn  1 &\nn  1 &\nnn  0 &\nn  0 &\nnn -1 &\nn  0 &\nnn  0 \\ \midrule
   0 &\nnn  1 &\nn  1 &\nnn -1 &\nn  0 &\nnn -1 &\nn -1 &\nnn  1 &\nn  0 &\nnn  0 \\
   1 &\nnn  0 &\nn  0 &\nnn -1 &\nn -1 &\nnn  0 &\nn  0 &\nnn  1 &\nn  0 &\nnn  0 \\ \midrule
   0 &\nnn  0 &\nn  0 &\nnn  0 &\nn  2 &\nnn -1 &\nn  0 &\nnn  0 &\nn -2 &\nnn  1 \\
   0 &\nnn  0 &\nn  0 &\nnn  0 &\nn  0 &\nnn  0 &\nn  0 &\nnn  0 &\nn  0 &\nnn  0 \\ \midrule
   2 &\nnn -1 &\nn -2 &\nnn  1 &\nn  0 &\nnn  0 &\nn  0 &\nnn  0 &\nn  0 &\nnn  0 \\
   0 &\nnn  0 &\nn  0 &\nnn  0 &\nn  0 &\nnn  0 &\nn  0 &\nnn  0 &\nn  0 &\nnn  0 \\ \midrule
   0 &\nnn  0 &\nn  0 &\nnn  0 &\nn  0 &\nnn  0 &\nn -1 &\nnn -1 &\nn  1 &\nnn  1 \\
   0 &\nnn  0 &\nn  0 &\nnn  0 &\nn  0 &\nnn  0 &\nn  0 &\nnn  0 &\nn  0 &\nnn  0 \\ \midrule
  -2 &\nnn  1 &\nn  0 &\nnn  0 &\nn  2 &\nnn -1 &\nn  0 &\nnn  0 &\nn  0 &\nnn  0 \\
  -2 &\nnn  1 &\nn  0 &\nnn  0 &\nn  2 &\nnn -1 &\nn  0 &\nnn  0 &\nn  0 &\nnn  0 \\ \midrule
   0 &\nnn  0 &\nn  1 &\nnn -1 &\nn  0 &\nnn  1 &\nn -1 &\nnn  1 &\nn  0 &\nnn -1 \\
   0 &\nnn  0 &\nn  0 &\nnn -1 &\nn  1 &\nnn  0 &\nn  0 &\nnn  1 &\nn -1 &\nnn  0 \\ \midrule
   0 &\nnn  0 &\nn  0 &\nnn -1 &\nn  1 &\nnn  0 &\nn  0 &\nnn  1 &\nn -1 &\nnn  0 \\
   0 &\nnn  0 &\nn  1 &\nnn -1 &\nn  0 &\nnn  1 &\nn -1 &\nnn  1 &\nn  0 &\nnn -1 \\ \midrule
   1 &\nnn  1 &\nn -1 &\nnn -1 &\nn  0 &\nnn  0 &\nn  0 &\nnn  0 &\nn  0 &\nnn  0 \\
   1 &\nnn  1 &\nn -1 &\nnn -1 &\nn  0 &\nnn  0 &\nn  0 &\nnn  0 &\nn  0 &\nnn  0 \\ \midrule
   0 &\nnn  0 &\nn  0 &\nnn  0 &\nn  0 &\nnn  0 &\nn -2 &\nnn  1 &\nn  2 &\nnn -1 \\
   0 &\nnn  0 &\nn  0 &\nnn  0 &\nn  0 &\nnn  0 &\nn -2 &\nnn  1 &\nn  2 &\nnn -1 \\ \midrule \\[-15pt] \midrule
   1 &\nnn  1 &\nn  0 &\nnn -1 &\nn -1 &\nnn  1 &\nn -1 &\nnn  0 &\nn  1 &\nnn -1 \\
  -1 &\nnn  2 &\nn  1 &\nnn -1 &\nn  0 &\nnn  1 &\nn  0 &\nnn -1 &\nn  0 &\nnn -1
   \end{array}
  \right]
  \;
  \left[
  \begin{array}{rr|rr|rr|rr|rr}
  1 &\!   0 &\!   0 &\!   0 &\!   0 &\!   0 &\!   0 &\!   0 &\nn -1 &\nn  0 \\
  0 &\!   1 &\!   0 &\!   0 &\!   0 &\!   0 &\!   0 &\!   0 &\nn  0 &\nn -1 \\
  0 &\!   0 &\!   1 &\!   0 &\!   0 &\!   0 &\!   0 &\!   0 &\nn -1 &\nn  0 \\
  0 &\!   0 &\!   0 &\!   1 &\!   0 &\!   0 &\!   0 &\!   0 &\nn  0 &\nn -1 \\
  0 &\!   0 &\!   0 &\!   0 &\!   1 &\!   0 &\!   0 &\!   0 &\nn -1 &\nn  0 \\
  0 &\!   0 &\!   0 &\!   0 &\!   0 &\!   1 &\!   0 &\!   0 &\nn  0 &\nn -1 \\
  0 &\!   0 &\!   0 &\!   0 &\!   0 &\!   0 &\!   1 &\!   0 &\nn -1 &\nn  0 \\
  0 &\!   0 &\!   0 &\!   0 &\!   0 &\!   0 &\!   0 &\!   1 &\nn  0 &\nn -1
  \end{array}
  \right]
  \]
  \end{minipage}
  \end{center}

  \smallskip
  \smallskip
  
\noindent
Further calculations show that for all $\lambda \vdash 4$ the ranks of $M_\lambda$ and $N_\lambda$ are equal:
  \[
  \begin{array}{l|rrrrr}
  \lambda & 4 & 31 & 22 & 211 & 1111 \\
  d_\lambda & 1 & 3 & 2 & 3 & 1 \\
  \text{rank} & 4 & 12 & 8 & 10 & 2
  \end{array}
  \]
We conclude that $f(x,y,z,t)$ belongs to the $S_4$-module generated by the consequences in degree 4 
of the linearized forms of the alternative laws.
\end{example}
  
%%%%%%%%%%%%%%%%%%%%%%%%%%%%%%%%%%%%%%%%%%%%%%%%%%%%%%%%%%%%%%%%%%%%%%%%%%%%%%%%%%%%%%%%%%%%%%%%%%%%%%%%%%%%

\subsection{Bondari's algorithm for finite-dimensional algebras}

Bondari \cite{Bondari1993,Bondari1997} introduced an algorithm using the representation theory of $S_n$ 
which computes an independent generating set for the multilinear identities (and central identities) of 
the full matrix algebra $M_k(\mathbb{F})$ with $\mathrm{char}\,\mathbb{F} = 0$ or  $\mathrm{char}\,\mathbb{F} = p > n$ 
where $n$ is the degree of the identities under consideration.
He constructed all the multilinear identities of degrees $\leq 8$ for $M_3(\mathbb{F})$, confirming existing results 
in the literature and discovering a new central identity in degree 8.

Bondari's algorithm can be used to find multilinear polynomial identities up to a certain degree
(depending on computational limitations) for any algebra $A$ over $\mathbb{F}$ of dimension $d < \infty$.
This algorithm involves evaluating matrix units in $\mathbb{F}S_n$ using the structure constants of $A$ with respect to
a chosen basis.

\begin{definition}
Fix $\lambda \vdash n$ and $f = f_1 + \dots + f_t \in (\mathbb{F} S_n)^t$.
The rank of the matrix $\phi_\lambda (f)$ is called the \textbf{rank of $f$ for $\lambda$}.
If this rank is 1, then we say that $f$ is \textbf{irreducible for $\lambda$}.
(That is, the isotypic component of type $\lambda$ in the submodule generated by $f$ is irreducible.)
\end{definition}

Consider $f \in (\mathbb{F} S_n)^t$ and let $r$ be the rank of the matrix $\mathrm{RCF}(\phi_\lambda( f ))$.
Each of the $r$ nonzero rows $g_1, \dots, g_r$ generates an irreducible submodule of type $\lambda$, 
and the isotypic component of type $\lambda$ is the direct sum of these $r$ isomorphic submodules; 
in other words, $r$ is the multiplicity of $\lambda$ in the submodule generated by $f$.
Extending Lemma \ref{firstrow} to the case of $t > 1$ association types, we see that each $g_i$ can
be regarded independently as an irreducible identity for $\lambda$ in the first row of the matrix.

\begin{lemma}
Every polynomial identity $f \in (\mathbb{F} S_n)^t$ is equivalent to a finite set of identities, 
each of which is irreducible for some $\lambda \vdash n$.
\end{lemma}

\begin{proof}
This is another way of saying that every finite dimensional $S_n$-module over $\mathbb{F}$ is the
direct sum of irreducible modules.
\end{proof}

Recall the images of the matrix units, $U^\lambda_{1j} = \psi(E^\lambda_{1j}) \in \mathbb{F} S_n$.
The general element $h \in \mathbb{F} S_n$ which is irreducible for $\lambda \vdash n$ has the form
  \[
  h = \sum_{k=1}^t \sum_{j=1}^{d_\lambda} x^k_{1j} [ U^\lambda_{1j} ]_k \qquad (x^k_{1j} \in \mathbb{F}).
  \]
Suppose that $A$ has basis $b_1, \dots, b_d$.
We describe one iteration of Bondari's algorithm.
We choose arbitrary elements $a_1, \dots, a_n \in A$ and evaluate the $[ U^\lambda_{1j} ]_k$:
  \[
  [ U^\lambda_{1j} ]_k( a_1, \dots, a_n ) = \sum_{i=1}^d c^i_{kj} b_i.
  \]
(This step can be very time-consuming, since the number of terms in the elements 
$U^\lambda_{1j} \in \mathbb{F}S_n$ is roughly $n!$.)
Combining the last two equations we obtain
  \[
  h( a_1, \dots, a_n ) = \sum_{i=1}^d \left[ \sum_{k=1}^t \sum_{j=1}^{d_\lambda} c^i_{kj} x^k_{1j} \right] b_i.
  \]
If $h$ is an identity for $A$ then the coefficient of each $b_i$ must be 0 for all $a_1, \dots, a_n \in A$:
  \[
  \sum_{k=1}^t \sum_{j=1}^{d_\lambda} c^i_{kj} x^k_{1j} = 0 \qquad ( 1 \le i \le d ).
  \]
This is a homogeneous linear system of $d$ equations in the $t d_\lambda$ coefficients $x^k_{1j}$ of the identity.
We compute the RCF of the coefficient matrix, and find its rank.

After $s$ iterations, we have a linear system of $sd$ equations.
We repeat this process until the rank stabilizes.
We then solve the system by computing the nullspace of the RCF.
The nonzero vectors in the nullspace are (probably) coefficient vectors of identities satisfied by $A$.
We need to check these identities by further computations.

%%%%%%%%%%%%%%%%%%%%%%%%%%%%%%%%%%%%%%%%%%%%%%%%%%%%%%%%%%%%%%%%%%%%%%%%%%%%%%%%%%%%%%%%%%%%%%%%%%%%%%%%%%%%

\subsection{Rational and modular arithmetic}

In general, we prefer to do all linear algebra computations over the field $\mathbb{Q}$ of rational numbers.
However, even if a matrix is very sparse and its entries are very small,
computing its RCF can produce exponential increases in the entries.
Even if enough computer memory is available to store the intermediate results, the calculations can take far too long.
It is therefore often convenient to use modular arithmetic, so that each entry uses a fixed small
amount of memory.
This leads to the problem of rational reconstruction: recovering correct results over $\mathbb{Q}$ or $\mathbb{Z}$
from known results over $\mathbb{F}_p$.

Rational reconstruction is not well-defined: we try to compute an inverse for a
partially-defined infinity-to-one map.
It is only effective when we have a good theoretical understanding of the expected results.
For our computations, Remark \ref{charpremark}
explains why we may assume that the correct rational coefficients 
have $n!$ as their common denominator where $n$ is the degree of the identities under consideration;
see also \cite[Lemma 8]{BP2009a}.
If we use a prime $p > n!$ then we can guess the common denominator
$b$ of the rational coefficients $a/b$ from the distribution of the congruence classes
modulo $p$: the modular coefficients are clustered near the congruence classes representing $a/b$ for $1 \le a \le b-1$.
This allows us to recover the rational coefficients; we then multiply by the LCM of the denominators to get
integer coefficients, and finally divide by the GCD of the coefficients.

Most of our computations require finding a basis of integer vectors for the nullspace of an integer matrix.
In some cases, modular methods give good results, meaning that
the basis vectors have small Euclidean lengths.
In other cases, we obtain much better results using the Hermite normal form (HNF) of an integer matrix
together with the LLL algorithm for lattice basis reduction.
If $M$ is an $s \times t$ matrix over $\mathbb{Z}$ then computing the HNF of the transpose produces two
matrices over $\mathbb{Z}$: a $t \times s$ matrix $H$ and a $t \times t$ matrix $U$ with $\det(U) = \pm 1$
such that $U M^t = H$.
If $\mathrm{rank}(M) = r$ then the bottom $t-r$ rows of $U$ form a lattice basis for the left integer nullspace of $M^t$,
which is the right integer nullspace of $M$.
We then apply the LLL algorithm to obtain shorter basis vectors.
For details, see \cite[\S3]{BP2009b}, \cite{BremnerLLL}.

We consider the fill and reduce algorithm in more detail.
By an error we mean that row reduction produces a row whose leading entry $a/b$ is nonzero in $\mathbb{Q}$ 
but zero in $\mathbb{F}_p$: in lowest terms $\gcd(a,b) = 1$ and $p \mid a$.
We assume that the probability of this is $1/p$. 
We can make this entry 1 using rational arithmetic, but it will be 0 using modular arithmetic. 
If the algebra $A$ has dimension $d$, then each iteration of the algorithm produces another $d$ linear equations
in the coefficients of the polynomial identity.
So we expect to perform $d$ operations of scalar multiplication of a row during the iteration. 
The chance that no error occurs is $[1-1/p]^d$.
The chance that an error occurs before the rank stabilizes, and remains for $s$ iterations after it stabilizes, is
$[ 1 - [ 1 - 1/p ]^d ]^s$.
For example, if we use $p = 101$ for an algebra of dimension $d = 8$ and perform $s = 10$
iterations after the rank has stabilized, then the probability of incorrect results is
$\approx 0.688 \cdot 10^{-11}$.

We conclude this section with an important special case.
Suppose that $f \equiv 0$ is an identity with rational coefficients satisfied by the algebra $A$ which
has integral structure constants with respect to a given basis. 
We multiply $f$ by the LCM of the denominators of its coefficients, 
obtaining a polynomial $f'$ with integral coefficients; 
we then divide $f'$ by the GCD of its coefficients, 
obtaining a polynomial $f''$ whose coefficients are integers with no common factor. 
It is clear that $f'' \equiv 0$ is an identity satisfied by $A$, and that the
reduction of $f''$ modulo $p$ is nonzero for any prime number $p$. 
Thus the existence of identities in characteristic 0 implies the existence of identities in characteristic $p$ for all $p$, 
and so non-existence in characteristic $p$ for a single prime $p$ implies non-existence in characteristic 0.
Therefore we can verify non-existence of identities over $\mathbb{Q}$ by computation over $\mathbb{F}_p$.

%%%%%%%%%%%%%%%%%%%%%%%%%%%%%%%%%%%%%%%%%%%%%%%%%%%%%%%%%%%%%%%%%%%%%%%%%%%%%%%%%%%%%%%%%%%%%%%%%%%%%%%%%%%%

\subsection{Polynomial identities of Cayley-Dickson algebras}

The most important alternative algebra is the division algebra $\mathbb{O}$ of real octonions,
which arises from the Cayley-Dickson doubling process 
$\mathbb{R} \subset \mathbb{C} \subset \mathbb{H} \subset \mathbb{O}$; see \cite[\S 2.2]{ZSSS1982}.
Cayley-Dickson algebras (also called generalized octonion algebras)
are 8-dimensional alternative algebras 
$C( \alpha, \beta, \gamma )$ 
%over a field $\mathbb{F}$ 
depending on parameters $\alpha, \beta, \gamma \in \mathbb{F} \setminus \{0\}$.
Kleinfeld classified simple alternative algebras in terms of Cayley-Dickson algebras.

\begin{theorem}
\label{Kleinfeld}
\emph{Kleinfeld, 1953 \cite{Kleinfeld1953}.}
A simple non-associative alternative algebra is a Cayley-Dickson algebra over its center.
\end{theorem}

If $ \mathbb{F} = \mathbb{R}$ then $C(-1, -1, -1) = \mathbb{O}$.
If $\mathrm{char}\,\mathbb F \ne 2$ then it is possible to choose a basis $1, e_1 , \dots , e_7$ of $C( \alpha, \beta, \gamma )$
so that its multiplication table is Table \ref{multiplicationoctonions}.

\begin{table}[h]
  \[
  \begin{array}{c|cccccccc}
  & 1 & e_1  & e_2 & e_3  & e_4  & e_5 & e_6 & e_7
  \\[2pt]
  \midrule
  1 & 1 & e_1  & e_2 & e_3  &  e_4 & e_5 & e_6 & e_7
  \\[2pt]
  e_1 & e_1 & \alpha  & e_3 & \alpha  e_2 & e_5 & \alpha  e_4 & -e_7 & -\alpha  e_6
  \\[2pt]
  e_2 & e_2 & - e_3 &  \beta  & - \beta  e_1 & e_6 & e_7  & \beta  e_4 & \beta  e_5
  \\[2pt]
  e_3 & e_3 & - \alpha  e_2 & \beta  e_1 & - \alpha  \beta  & e_7 &  \alpha  e_6 & - \beta  e_5 & - \alpha  \beta  e_4
  \\[2pt]
  e_4 & e_4 &  - e_5 & - e_6 & - e_7 &  \gamma  &  - \gamma  e_1 & - \gamma  e_2 & - \gamma  e_3
  \\[2pt]
  e_5 & e_5 & - \alpha  e_4 & - e_7 & -\alpha  e_6 &  \gamma  e_1 & - \alpha  \gamma  & \gamma  e_3  &  \alpha  \gamma  e_2
  \\[2pt]
  e_6 & e_6 & e_7 & - \beta  e_4 & \beta  e_5 & \gamma  e_2 & -\gamma  e_3 & -\beta  \gamma & - \beta  \gamma  e_1
  \\[2pt]
  e_7 & e_7 & \alpha  e_6 & - \beta  e_5 & \alpha  \beta  e_4 & \gamma  e_3 &  - \alpha  \gamma  e_2 &  \beta  \gamma  e_1 & \alpha  \beta  \gamma
  \\
  \midrule
  \end{array}
  \]
  \caption{Multiplication table of the generalized octonions}
  \label{multiplicationoctonions}
\end{table}

\begin{problem}
Find a basis for the $T$-ideal of polynomial identities of a Cayley-Dickson algebra $C$.
\end{problem}

Isaev \cite{Isaev1984} found a finite basis of $T(C)$ when $\mathbb{F}$ is finite. 
Iltyakov \cite{Iltyakov1985} proved that $T(C)$ is finitely generated when $\mathrm{char}\,\mathbb{F} = 0$
but did not give a set of generators. 
Racine \cite{Racine1988} found the identities of degrees $\leq 5$ for $C$ when $\mathrm{char}\,\mathbb{F} \neq 2,3,5$.

Cayley-Dickson algebras are quadratic algebras, in the sense that they are unital algebras $C$ over $\mathbb F$ 
such that every $x \in C$ satisfies $x^2 - t(x) x + n(x) 1 = 0$, where the trace $t\colon C \to \mathbb{F}$ is a linear map 
and the norm $n\colon C \to \mathbb{F}$ is a quadratic form. 
If $x = a \cdot 1 + \sum_{i=1}^7 a_i e_i$ and ${\overline x} = a \cdot 1 - \sum_{i=1}^7 a_i e_i$ 
are an element of $C$ and its conjugate, then the trace and the norm of $x$ are as follows:
  \begin{align*}
  t(x) &= x + {\overline x} = 2a, 
  \\
  n(x) &= x {\overline x} = a^2 - \alpha  a^2_1 - \beta a^2_2 + \alpha \beta a^2_3 - \gamma a^2_4 
  + \alpha \gamma a^2_5 + \beta \gamma a^2_6 - \alpha \beta \gamma a^2_7.
  \end{align*}

\begin{theorem} \label{racinetheorem}
\emph{Racine, 1985 \cite{Racine1985}.}
Every quadratic algebra satisfies the identity
  \[
  V(t^2) - V(t) \circ t \equiv 0,
  \]
where $x\circ y = xy + yx$, 
$V_x (y) = x \circ y$, and
$V = \sum_{\sigma\in S_3} \epsilon(\sigma) V_{x^\sigma} V_{y^\sigma} V_{z^\sigma}$.
%$V  =  V_x V_y V_z + V_z V_x V_y + V_y V_z V_x - V_y V_x V_z - V_x V_z V_y - V_z V_y V_x$.
\end{theorem}

It follows that every Cayley-Dickson algebra satisfies this identity.
The identities of degree $\leq 6$ satisfied by Cayley-Dickson algebras were found by Hentzel and Peresi 
using Bondari's algorithm. 

\begin{theorem} \label{HP97}
\emph{Hentzel and Peresi, 1997 \cite{HP1997}.}
The identities of degree $\le 6$ of Cayley-Dickson algebras are as follows, where
either $\mathrm{char}\,\mathbb{F} = 0$ or $\mathrm{char}\,\mathbb{F} = p > n$, and $n$
is the degree of the identity:
  \begin{alignat*}{2}
  n &\le 2 &\qquad &\text{no identities}
  \\
  n &= 3 &\qquad &(x,x,y) \equiv 0, \quad (x,y,y) \equiv 0 \quad \text{(alternative laws)}
  \\
  n &= 4 &\qquad &\text{no identities}
  \\
  n &= 5 &\qquad &V(t^2) - V(t) \circ t \equiv 0, \quad [[x,y] \circ [z,t] , w] \equiv 0
  \\
  n &= 6 &\qquad 
  &\Big[ 
  \sum_{\sigma \in S_5} \sgn(\sigma) \big(  24 x(y(z(tw))) + 8 x((y,z,t)w) - 11 (x,y,(z,t,w)) \big), \, 
  u \,
  \Big] 
  \equiv 0,
  \\
  &&&\text{where $\sigma$ permutes $x,y,z,t,w$ and $\sgn$ is the sign.}
  \end{alignat*}
We give only the identities which are not consequences of those of lower degrees.
\end{theorem}

In characteristic 0, Shestakov and Zhukavets \cite{SZ2009} found a basis of three identities 
(one of degree 5 and two of degree 6) for the skew-symmetric identities of $\mathbb{O}$. 
In characteristic $\ne 2, 3, 5$, 
Shestakov \cite{Shestakov2011} found a basis of identities for split Cayley-Dickson algebras $C$
modulo the associator ideal of a free alternative algebra; 
that is, a basis for a homomorphic image $T'(C)$ in the free associative algebra
of the $T$-ideal $T(C)$ of identities of $C$.  
Henry \cite{Henry} found a basis for the ${\mathbb Z}_2^2$-graded and ${\mathbb Z}_2^3$-graded identities for 
Cayley-Dickson algebras (the latter case requires characteristic $\ne 2$).
Bremner and Hentzel \cite{BH2002} studied identities for alternative algebras which are built out of associators;
in degree 7, they found two identities satisfied by the associator in every alternative algebra, 
and five identities satisfied by the associator in $\mathbb{O}$.

%%%%%%%%%%%%%%%%%%%%%%%%%%%%%%%%%%%%%%%%%%%%%%%%%%%%%%%%%%%%%%%%%%%%%%%%%%%%%%%%%%%%%%%%%%%%%%%%%%%%%%%%%%%%

\subsection{Multilinear identities for the octonions}

We apply the computational techniques described in previous sections to the multilinear polynomial identities 
satisfied by the algebra $\mathbb{O}$ of octonions.
We recover all the existing results in the literature on identities in degree $\le 6$,
and then show that there are no new identities in degree 7.
As basis for $\mathbb{O}$ over the field $\mathbb{F}$ we take the symbols $1, e_1, \dots, e_7$.
The structure constants depend on parameters $\alpha,\beta,\gamma \in \mathbb{F}$; 
see Table \ref{multiplicationoctonions}.
If $\mathbb{F} = \mathbb{R}$ and $\alpha=\beta=\gamma=-1$ then we obtain the alternative division algebra of real octonions,
which is the case we consider in what follows.
For the alternative laws, their linearizations, and their consequences in degree 4, see Example \ref{alternative4}.

\subsubsection*{Degree 3}

Every multilinear identity in degree 3 satisfied by $\mathbb{O}$ follows from the linearizations of the alternative laws;
see \cite[\S 9, Example 1]{BMS2007}.

\subsubsection*{Degree 4}

Every multilinear identity of degree 4 satisfied by $\mathbb{O}$ follows from the consequences of the alternative laws;
see \cite{Racine1988}.
We will verify this result using our computational methods.
The partitions $\lambda \vdash 4$ are 4, 31, 22, 211, 1111 with corresponding dimensions $d_\lambda =$ 1, 3, 2, 3, 1.
The $t = 5$ association types are 
  \[
  ((\ast\ast)\ast)\ast, \qquad 
  (\ast(\ast\ast))\ast, \qquad 
  (\ast\ast)(\ast\ast), \qquad 
  \ast((\ast\ast)\ast), \qquad 
  \ast(\ast(\ast\ast)).
  \]
We give details for $\lambda = 22$; the other cases are similar.
The standard tableaux are:
  \[
  \young(12,34) \qquad\qquad \young(13,24)
  \]
The elements $U^\lambda_{11}, U^\lambda_{12} \in \mathbb{Q} S_4$ corresponding to the first row matrix units are
  \begin{align*}
  U^\lambda_{11}
  =
  \psi( E^\lambda_{11} )
  &=
  1234
  - 1432
  - 3214
  + 3412
  + 1243
  - 1342
  - 4213
  + 4312
  \\
  &\quad
  + 2134
  - 2431
  - 3124
  + 3421
  + 2143
  - 2341
  - 4123
  + 4321,
  \\
  U^\lambda_{12}
  =
  \psi( E^\lambda_{12} )
  &=
  1324
  - 1342
  - 3124
  + 3142
  + 1423
  - 1432
  - 4123
  + 4132
  \\
  &\quad
  + 2314
  - 2341
  - 3214
  + 3241
  + 2413
  - 2431
  - 4213
  + 4231.
  \end{align*}
We create an $18 \times 10$ matrix consisting of $2 \times 2$ blocks,
with a $10 \times 10$ upper block and an $8 \times 10$ lower block.
The columns correspond to the following elements of the direct sum of $t = 5$ copies of $\mathbb{F} S_4$, 
where the subscripts give the association types:
  \[
  [ U^\lambda_{11} ]_1 \;\;
  [ U^\lambda_{12} ]_1 \quad
  [ U^\lambda_{11} ]_2 \;\;
  [ U^\lambda_{12} ]_2 \quad
  [ U^\lambda_{11} ]_3 \;\;
  [ U^\lambda_{12} ]_3 \quad
  [ U^\lambda_{11} ]_4 \;\;
  [ U^\lambda_{12} ]_4 \quad
  [ U^\lambda_{11} ]_5 \;\;
  [ U^\lambda_{12} ]_5
  \]
Any identity for $\mathbb{O}$ of type $\lambda$ can be expressed as a linear combination of these elements.
The fill-and-reduce algorithm converges after one iteration to this matrix:
  \[
  \left[
  \begin{array}{rr|rr|rr|rr|rr}
  1 & 0 & 1 & 0 & 1 & 0 & 1 & 0 & 1 & 0 \\
  0 & 1 & 0 & 1 & 0 & 1 & 0 & 1 & 0 & 1
  \end{array}
  \right]
  \]
We find a basis for the nullspace and calculate its RCF, obtaining the matrix
whose rows represent identities of type $\lambda$ satisfied by $\mathbb{O}$:
  \[
  \texttt{allmat}(\lambda)
  =
  \left[
  \begin{array}{rr|rr|rr|rr|rr}
  1 & 0 & 0 & 0 & 0 & 0 & 0 & 0 & -1 & 0 \\
  0 & 1 & 0 & 0 & 0 & 0 & 0 & 0 & 0 & -1 \\
  0 & 0 & 1 & 0 & 0 & 0 & 0 & 0 & -1 & 0 \\
  0 & 0 & 0 & 1 & 0 & 0 & 0 & 0 & 0 & -1 \\
  0 & 0 & 0 & 0 & 1 & 0 & 0 & 0 & -1 & 0 \\
  0 & 0 & 0 & 0 & 0 & 1 & 0 & 0 & 0 & -1 \\
  0 & 0 & 0 & 0 & 0 & 0 & 1 & 0 & -1 & 0 \\
  0 & 0 & 0 & 0 & 0 & 0 & 0 & 1 & 0 & -1
  \end{array}
  \right]
  \]
Using Clifton's algorithm we obtain the matrix representing 
the 10 consequences in degree 4 of the alternative laws for partition $\lambda$; 
this is $M_\lambda = \phi_\lambda(f^1,\dots,f^{10})$ from Example \ref{altdeg4}, 
whose RCF equals $\texttt{allmat}(\lambda)$.

\subsubsection*{Degree 5}

Racine \cite{Racine1988} found two new polynomial identities in degree 5 for $\mathbb{O}$:
  \[
  \mathrm{(R1)} \quad [ [x,y]^2, x ] \equiv 0, 
  \qquad\qquad\qquad 
  \mathrm{(R2)} \quad V(t^2) - V(t) \circ t \equiv 0,
  \]
where $[x,y] = xy - yx$, $x \circ y = xy + yx$, and the square is with respect to the multiplication in $\mathbb{O}$.
For the definition of the operator $V$, see Theorem \ref{racinetheorem}.

\begin{remark} 
The multilinear form of the identity (R2) can be written as
  \[
  x^2 s_3^+(y,z,t) - xs_3^+(y,z,t) \circ x \equiv 0,
  \]
where
$s_3^+(x,y,z) = s_3(R_\circ(x),R_\circ(y),R_\circ(z))$ is an operator acting on the right,
$s_3$ is the standard polynomial of degree 3,
and $R_\circ(y)$ the (right) multiplication operator by $y$ using $\circ$: $x R_\circ(y) = x \circ y$;
this follows the notation of \cite{Racine1988}.
\end{remark}

(R1) and (R2) are satisfied by $\mathbb{O}$, but are not quite sufficient to generate $\mathrm{New}(5)$.
Hentzel and Peresi \cite{HP1997} proved that $[v,w] \circ [x,y]$ is a central polynomial; that is,
  \begin{equation}
  [ [v,w] \circ [x,y], z] \equiv 0,
  \tag{HP5}
  \end{equation}
is an identity of degree 5 for the algebra of octonions.
The $S_5$-module $\mathrm{New}(5)$ is generated by (HP5) and (R2).
Using our computational techniques, we obtained the results summarized in Table \ref{deg5table}.
Column $r_\text{all}$ gives the multiplicity of the irreducible $S_5$-module $[\lambda]$ in the module of all multilinear
identities satisfied by $\mathbb{O}$.
Column $r_\text{old}$ gives the multiplicity of $[\lambda]$ in the module of all consequences
of the alternative laws.
Column $r_\text{old+R1+R2}$ gives the multiplicity of $[\lambda]$ in the module generated by the consequences 
of the alternative laws and the two Racine identities (R1) and (R2).
From this we see that (R1) and (R2) are sufficient in the first four representations, 
but in each of the last three representations, the multiplicities are one less than required.
Column $r_\text{old+R2+HP5}$ gives the multiplicity of $[\lambda]$ in the module generated by the consequences 
of the alternative laws together with the identities (R2) and (HP5)
these values are the same as $r_\text{all}$ for all $\lambda$, and the corresponding matrices are equal.
The last two columns verify that, modulo the consequences of the alternative laws, 
neither of the identities (R2) or (HP5) generates $\mathrm{New}(5)$ by itself, 
and that these two identities are independent (neither is implied by the other).

  \begin{table}
  \begin{tabular}{lr|cccc|cc}
  $\lambda$ & $d_\lambda$ &\qquad $r_\text{all}$ &\quad $r_\text{old}$ & $r_\text{old+R1+R2}$ & $r_\text{old+R2+HP5}$ &
  $r_\text{old+R2}$ & $r_\text{old+HP5}$ 
  \\ \midrule
  5 & 1 &\qquad 13 &\quad 13 & 13 & 13 & 13 & 13 
  \\
  41 & 4 &\qquad 52 &\quad 52 & 52 & 52 & 52 & 52
  \\
  32 & 5 &\qquad 66 &\quad 65 & 66 & 66 & 65 & 66
  \\
  311 & 6 &\qquad 76 &\quad 75 & 76 & 76 & 76 & 75
  \\
  221 & 5 &\qquad 64 &\quad 63 & 63 & 64 & 63 & 64
  \\
  2111 & 4 &\qquad 48 &\quad 46 & 47 & 48 & 47 & 47 
  \\
  11111 & 1 &\qquad 11 &\quad 10 & 10 & 11 & 10 & 11
  \\ \midrule
  \end{tabular}
  \caption{Multiplicities of irreducible modules in degree 5}
  \label{deg5table}
  \end{table}

  \begin{figure}
  \begin{center}
  \small
  \begin{align*}
  \mathrm{allmat}(\lambda)
  &=  
  \left[
  \begin{array}{rrrrrrrrrrrrrr}
  1 & 0 & 0 & 0 & 0 & 0 & 0 & 0 &\;\;\, 0 & 0 & 0 & -2 &  1 &  0 \\
  0 & 1 & 0 & 0 & 0 & 0 & 0 & 0 & 0 & 0 & 0 & -2 &  0 &  1 \\
  0 & 0 & 1 & 0 & 0 & 0 & 0 & 0 & 0 & 0 & 0 & -1 &  0 &  0 \\
  0 & 0 & 0 & 1 & 0 & 0 & 0 & 0 & 0 & 0 & 0 & -3 &  1 &  1 \\
  0 & 0 & 0 & 0 & 1 & 0 & 0 & 0 & 0 & 0 & 0 &  0 & -1 &  0 \\
  0 & 0 & 0 & 0 & 0 & 1 & 0 & 0 & 0 & 0 & 0 & -1 &  0 &  0 \\
  0 & 0 & 0 & 0 & 0 & 0 & 1 & 0 & 0 & 0 & 0 & -1 &  0 &  0 \\
  0 & 0 & 0 & 0 & 0 & 0 & 0 & 1 & 0 & 0 & 0 & -1 &  0 &  0 \\
  0 & 0 & 0 & 0 & 0 & 0 & 0 & 0 & 1 & 0 & 0 & -1 &  0 &  0 \\
  0 & 0 & 0 & 0 & 0 & 0 & 0 & 0 & 0 & 1 & 0 & -2 &  0 &  1 \\
  0 & 0 & 0 & 0 & 0 & 0 & 0 & 0 & 0 & 0 & 1 &  1 & -1 & -1
  \end{array}
  \right]
  \\
  \mathrm{oldmat}(\lambda)
  &=  
  \left[
  \begin{array}{rrrrrrrrrrrrrr}
  1 & 0 & 0 & 0 & 0 & 0 & 0 & 0 & -3 & 0 & 0 &  1 &  1 &  0 \\
  0 & 1 & 0 & 0 & 0 & 0 & 0 & 0 & -1 & 0 & 0 & -1 &  0 &  1 \\
  0 & 0 & 1 & 0 & 0 & 0 & 0 & 0 & -2 & 0 & 0 &  1 &  0 &  0 \\
  0 & 0 & 0 & 1 & 0 & 0 & 0 & 0 & -2 & 0 & 0 & -1 &  1 &  1 \\
  0 & 0 & 0 & 0 & 1 & 0 & 0 & 0 & -1 & 0 & 0 &  1 & -1 &  0 \\
  0 & 0 & 0 & 0 & 0 & 1 & 0 & 0 & -1 & 0 & 0 &  0 &  0 &  0 \\
  0 & 0 & 0 & 0 & 0 & 0 & 1 & 0 &  0 & 0 & 0 & -1 &  0 &  0 \\
  0 & 0 & 0 & 0 & 0 & 0 & 0 & 1 & -2 & 0 & 0 &  1 &  0 &  0 \\
  0 & 0 & 0 & 0 & 0 & 0 & 0 & 0 &  0 & 1 & 0 & -2 &  0 &  1 \\
  0 & 0 & 0 & 0 & 0 & 0 & 0 & 0 &  0 & 0 & 1 &  1 & -1 & -1
  \end{array}
  \right]
  \end{align*}
  \vspace{-6mm}
  \end{center}
  \caption{Matrices for new octonion identities ($\lambda = 11111$)}
  \label{octdeg5mat}
  \end{figure}

We conclude this discussion by presenting explicit matrices to illustrate how we can obtain 
new identities from the matrix units in the group algebra. 
For the last partition $\lambda = 11111$ with dimension $d_\lambda = 1$, 
we obtain the matrices $\mathrm{allmat}(\lambda)$ and $\mathrm{oldmat}(\lambda)$
displayed in Figure \ref{octdeg5mat}, with ranks of 11 and 10 respectively.
The row space of $\mathrm{oldmat}(\lambda)$ is a subspace of the row space of $\mathrm{allmat}(\lambda)$.
Row 9 of $\mathrm{allmat}(\lambda)$ has a leading 1 in column 9, but $\mathrm{oldmat}(\lambda)$ 
has no leading 1 in this column.
Therefore row 9 of $\mathrm{allmat}(\lambda)$ represents an identity satisfied by $\mathbb{O}$
which is not a consequence of the alternative laws.
In terms of matrix units, this row is $E_{9,9} - E_{9,12}$ and therefore represents the following identity:
  \[
  \sum_{\sigma \in S_5} \epsilon(\sigma) 
  \Big[ 
  ( x_{\sigma(1)} x_{\sigma(2)} ) ( x_{\sigma(3)} ( x_{\sigma(4)} x_{\sigma(5)} ) )
  -
  x_{\sigma(1)} ( ( x_{\sigma(2)} x_{\sigma(3)} ) ( x_{\sigma(4)} x_{\sigma(5)} ) )
  \Big] 
  \equiv 0.
  \]

\subsubsection*{Degree 6}

  \begin{table}
  \begin{tabular}{lr|rrr}
  $\lambda$ & $d_\lambda$ &\qquad $r_\text{all}$ &\qquad $r_\text{alt}$ &\qquad $r_\text{old}$
  \\ \midrule
  6 & 1 &\qquad 41 &\qquad 41 &\qquad 41
  \\
  51 & 5 &\qquad 205 &\qquad 205 &\qquad 205
  \\
  42 & 9 &\qquad 372 &\qquad 369 &\qquad 372
  \\
  411 & 10 &\qquad 409 &\qquad 406 &\qquad 409
  \\
  33 & 5 &\qquad 207 &\qquad 205 &\qquad 207
  \\
  321 & 16 &\qquad 660 &\qquad 652 &\qquad 660
  \\
  3111 & 10 &\qquad 407 &\qquad 400 &\qquad 407
  \\
  222 & 5 &\qquad 204 &\qquad 202 &\qquad 204
  \\
  2211 & 9 &\qquad 368 &\qquad 360 &\qquad 368
  \\
  21111 & 5 &\qquad 202 &\qquad 194 &\qquad 202
  \\ 
  111111 & 1 &\qquad 40 &\qquad 36 &\qquad 39
  \\ \midrule
  \end{tabular}
  \caption{Multiplicities of irreducible modules in degree 6}
  \label{deg6table}
  \end{table}
  
Hentzel and Peresi \cite{HP1997} discovered a multilinear central polynomial of degree 5 for $\mathbb{O}$,
which produces the following polynomial identity where $(x,y,z)=(xy)z-x(yz)$ is the associator, and $S_5$ permutes
$x,y,z,t,w$:
  \begin{equation}
  \tag{HP6}
  \Big[ 
  \sum_{\sigma \in S_5} \sgn(\sigma) \big(  24 x(y(z(tw))) + 8 x((y,z,t)w) - 11 (x,y,(z,t,w)) \big), \, 
  u \,
  \Big] 
  \equiv 0.
  \end{equation}
Shestakov and Zhukavets \cite{SZ2009} found a somewhat simpler central polynomial  which produces the
following polynomial identity:
  \begin{equation}
  \tag{SZ}
  \Big[ 
  \sum_{\sigma \in S_5} \sgn(\sigma) \big( 12 ([x,y][z,t])w - [[[[x,y],z],t],w] \big), \,
  u \,
  \Big]
  \equiv 0.
  \end{equation}
Using our computational techniques, we obtained the results in Table \ref{deg6table}.
Column $r_\text{all}$ gives the multiplicity of the irreducible $S_6$-module $[\lambda]$ in the module of all multilinear
identities satisfied by $\mathbb{O}$.
Column $r_\text{alt}$ gives the multiplicity of $[\lambda]$ in the module of all consequences
of the alternative laws.
Column $r_\text{old}$ gives the multiplicity of $[\lambda]$ in the module generated by the consequences 
of the alternative laws and the identities (R2) and (HP5).
From this we see that $r_\text{old} = r_\text{all}$ except for $\lambda = 111111$ where the difference is 1;
hence there is a new identity which alternates in all 6 variables.
We further checked that the multiplicities for the alternative laws, (R2) and (HP5) together with either (HP6) or (SZ)
are equal to $r_\text{all}$ for all $\lambda$; hence either (HP6) or (SZ) can be taken as the new generator in degree 6.

Our computations led us to the following new identity in degree 6, which involves only two of the 42 association types,
alternates in all 6 variables, and does not have the form $[ f(v,w,x,y,z), u ] \equiv 0$ where $f$ is a 
central polynomial:
  \begin{equation}
  \label{newidentity6}
  \sum_{\sigma \in S_6} \sgn(\sigma) 
  \Big(
  5
  x_1 ( x_2 ( ( x_3 x_4 ) ( x_5 x_6 ) ) )
  -
  x_1 ( x_2 ( x_3 ( x_4 ( x_5 x_6 ) ) ) )
  \Big)
  \equiv 0.
  \end{equation}
We can use this identity instead of (HP6) or (SZ) as the new generator in degree 6.

\subsubsection*{Degree 7}

Our computations indicate that there are no new identities in degree 7.

\begin{theorem}
Every multilinear polynomial identity of degree $\le 7$ satisfied by the octonion algebra $\mathbb{O}$ 
is implied by the consequences of the alternative laws, the identities \emph{(R2)} and \emph{(HP5)}, 
and either \emph{(HP6)} or \emph{(SZ)} or identity \eqref{newidentity6}.
\end{theorem}

We therefore conclude this paper with the following conjecture.

\begin{conjecture}
The alternative laws together with the identities \emph{(R2)}, \emph{(HP5)}, and either \emph{(HP6)} or \emph{(SZ)} 
or \eqref{newidentity6}, generate the $T$-ideal of polynomial identities satisfied by the octonion algebra $\mathbb{O}$.
\end{conjecture}

%%%%%%%%%%%%%%%%%%%%%%%%%%%%%%%%%%%%%%%%%%%%%%%%%%%%%%%%%%%%%%%%%%%%%%%%%%%%%%%%%%%%%%%%%%%%%%%%%%%%%%%%%%%%

\section*{Acknowledgements}

Murray Bremner was supported by a Discovery Grant from NSERC, the Natural Sciences and
Engineering Research Council of Canada.
Sara Madariaga was supported by a Postdoctoral Fellowship from PIMS, the Pacific Institute
for the Mathematical Sciences.
Luiz Peresi thanks the Department of Mathematics and Statistics at the University of Saskatchewan
for its hospitality and financial support during his visits in summer 2012 and spring 2014.

%%%%%%%%%%%%%%%%%%%%%%%%%%%%%%%%%%%%%%%%%%%%%%%%%%%%%%%%%%%%%%%%%%%%%%%%%%%%%%%%%%%%%%%%%%%%%%%%%%%%%%%%%%%%

\end{document}